\documentclass[12pt,reqno,a4paper]{amsart}
\usepackage{latexsym, amsmath, amssymb,amsthm,amsopn,amsfonts}
\usepackage{version}
\usepackage{epsfig,graphics,color,graphicx,graphpap}
\usepackage{amssymb}

\usepackage[left=2.5cm,right=2.5cm,top=3cm,bottom=3cm]{geometry}

\newcommand{\ssi}{\scriptscriptstyle}
\newcommand{\iss}[1]{{\scriptscriptstyle (#1)}}

\numberwithin{equation}{section}
\newtheorem{theorem}{Theorem}[section]

\newtheorem{lemma}[theorem]{Lemma}
\newtheorem{definition}{Definition}[section]

\newcommand{\R}{\mathbb R}

\author[Yi-Hsuan Lin]{Yi-Hsuan Lin}
\address{Institute for Advanced Study, The Hong Kong University Science and Technology, Clear Water Bay, Kowloon, Hong Kong}
\curraddr{}
\email{yihsuanlin3@gmail.com}

\author[Shixu Meng]{Shixu Meng}
\address{Department of Mathematics, University of Michigan, MI 48109 USA}
\curraddr{}
\email{shixumen@umich.edu}

\begin{document}

\title[Homogenization problem in elastic scattering ]{Leading and second order  Homogenization of an elastic scattering problem for highly oscillating anisotropic medium}

\maketitle

\begin{abstract}

We consider the scattering of elastic waves by highly oscillating anisotropic periodic media with bounded support. Applying the two-scale homogenization, we first obtain a constant coefficient second-order partial differential elliptic equation that describes the wave propagation of the effective or overall wave field. We study the rate of convergence by introducing complimentary boundary correctors. To account for dispersion induced by the periodic structure, we further pursue a higher-order homogenization. We then investigate the rate of convergence and formally obtain a fourth-order differential equation that demonstrates the anisotropic dispersion.

	\medskip


\end{abstract}

\tableofcontents

\section{Introduction and summary of results}\label{Section 1}

\subsection{Motivation and background}
The wave propagation in periodic media is of great interest in cloaking, sub-wavelength imaging, and noise control, thanks to the underpinning phenomena of frequency-dependent anisotropy and band gaps~\cite{maldovan2013sound,milton2006cloaking,zhu2011holey}. Away from averaging techniques, the effective wave motion can also be obtained using the two-scale method \cite{alain2011asymptotic} with a perturbation parameter that signifies the ratio between the unit cell of periodicity and wavelength. In the regime of long-wavelength and low-frequency, the leading-order homogenization in particular gives the quasi-static model by a second-order partial differential equation, where the elastic tensor and density are replaced by constant effective elastic tensor and density respectively \cite{alain2011asymptotic, shen2017lectures}. To gain further understanding of the wave dispersion, a higher-order homogenization has been taken into account. A higher-order homogenization was derived for the scalar wave equation in \cite{chen2001dispersive, meng2017dynamic, wautier2015second}, where a fourth-order partial differential equation was formally derived and the dispersive effect was hence demonstrated. \cite{capdeville2007second} considered a higher-order homogenization of the elastic wave for non-periodic layered media that transcends the usual quasi-staic regime. Alternatively a dispersive model for scalar wave equation was derived using Floquet-Bloch theory and higher-order asymptotic of the Bloch variety \cite{santosa1991dispersive}.  In the case that the periodic structure was only supported in a bounded domain, contrary to the case that the periodic structure occupies $\R^d$ for $d=2,3$, the boundary correctors played a role both in the leading order and second order homogenization \cite{cakoni2016homogenization}. The higher-order homogenization in particular sheds light on  sensing the microstructure through dispersion \cite{lambert2015bridging}.  The time-harmonic elastic wave equation luckily can be handled by the homogenization theory for second-order elliptic systems. 

The homogenization problem for the second-order elliptic equations
or systems with periodic coefficients has been well studied in the literature, for example, see \cite{avellaneda1987compactness,avellaneda1989compactness,alain2011asymptotic,cioranescu2000introduction,kenig2012convergence,kenig2014periodic,pavliotis2008multiscale,shen2017lectures}.
Concerning about the regularity estimates, the authors in  \cite{avellaneda1987compactness,avellaneda1989compactness}
introduced a famous three-steps compactness method to prove the H\"older estimates for solutions of divergence and non-divergence elliptic
systems. Furthermore, in \cite{avellaneda1987compactness}, the authors studied the Green function and the Poisson kernel to establish $L^{p}$ theory of the elliptic homogenization problem. The authors in \cite{kenig2014periodic} investigated the asymptotic behaviour of the Green
and Neumann functions and derived optimal convergence rates in $L^{p}$
and $W^{1,p}$ for solutions with Dirichlet or Neumann boundary conditions.
They further studied the convergence
rates in $L^{2}$ of solutions of the elliptic systems in Lipschitz
domains in \cite{kenig2012convergence}. We refer to \cite{shen2017lectures} for an excellent lecture note for the survey on this research area. 

\subsection{The elastic scattering problem}
Let $D\subset\mathbb{R}^{d}$ be a bounded simply connected domain
with $C^{\infty}$-smooth boundary $\partial D$ for $d=2,3$. Let
$\epsilon>0$ be a small parameter and $Y:=[0,1]^{d}$ be the unit
cell. Let $C=C(y)$ be an anisotropic elastic fourth-order tensor with $C=(C_{ijk\ell})_{1\leq i,j,k,\ell \leq d}$. 
In this paper, we assume that the elastic tensor $C=C(y)$ satisfies the following conditions.

\begin{itemize}
\item  Periodicity: The elastic tensor $C=C(y)$ is  $Y$-periodic,  
\[
C(y+z)=C(y),\mbox{ for any }y\in\mathbb{R}^{d}\mbox{ and }z\in\mathbb{Z}^{d}.
\]

\item  Strong convexity:
\begin{align}
\label{strong convexity}\sum_{i,j,k,\ell=1}^{d}C_{ijk\ell}(y)a_{ij}a_{k\ell}\geq c_{0}\sum_{i,j=1}^{d}a_{ij}^{2},\text{ for any }y\in\mathbb R^d,
\end{align}
with some constant $c_{0}>0$ and for any constant symmetric matrix
$(a_{ij})_{1\leq i,j\leq d}$.

\item  Smoothness: $C(y)\in C^{\infty}(\mathbb R^d)$. 

\item  Symmetry:
The elastic tensor $C=(C_{ijk\ell})_{1\leq i,j,k,\ell \leq d}$ satisfies major and minor symmetric condition, that is, 
\begin{align}\label{Symmetry condition}
C_{ijk\ell }=C_{k\ell ij}\text{ and }C_{ijk\ell}=C_{ij\ell k},\text{ for all }1\leq i,j,k,\ell \leq d.
\end{align} 

\end{itemize}

The symmetric property of the elastic tensor $C=C(y)$ plays an important role in the study of the asymptotic analysis of the scattering homogenization problem (see Section \ref{Section 2} and Section \ref{Section 3}).
Next, let $\rho =\rho(y)\in C^{\infty}(\mathbb R^d)$ 
be the density of the medium and  $\omega\in\mathbb{R}$ be the interrogating
frequency. We also assume  $\rho(y)$ is $Y$-periodic, i.e., $\rho(y+z)=\rho(y)$ for any $y\in \mathbb R^d$ and $z\in \mathbb Z^d$. 
In the exterior domain $\mathbb{R}^{d}\backslash\overline{D}$, the medium is homogeneous, isotropic where $\rho=1$ and the elastic tensor is a constant fourth-order tensor $C^{\iss 0}$ given by
\begin{equation} 
\label{constant elastic tensor}C^{\iss 0}_{ijk\ell }=\lambda\delta_{ij}\delta_{k\ell}+\mu (\delta_{ik}\delta_{j\ell}+\delta_{i\ell}\delta_{jk}),
\end{equation}
where $\lambda$ and $\mu$ are  Lam\'e constants satisfying
the strong convexity condition \eqref{strong convexity}, and it is equivalent to
\[
\mu>0\quad \mbox{ and }\quad d\lambda+2\mu>0\quad \text{ where }\quad d=2,3.
\]

Now let us consider the elastic scattering by highly oscillating periodic media with bounded support. Let 
$u(x)=(u_{\ell}(x))_{\ell=1}^{d}$ be the displacement vector field, the time-harmonic elastic scattering is modeled by 
\begin{equation}
\begin{cases}
\nabla\cdot\left(C(\dfrac{x}{\epsilon})\nabla u\right)+\omega^{2}\rho(\dfrac{x}{\epsilon})u=0 & \mbox{ in }D,\\
\Delta^{*}u^{s}+\omega^{2}u^{s}=0 & \mbox{ in }\mathbb{R}^{d}\backslash\overline{D},\\
u^{s}+u^{i}=u & \mbox{ on }\partial D,\\
T_{\nu}(u^{s}+u^{i})=(C(\dfrac{x}{\epsilon})\nabla u)\cdot\nu & \mbox{ on }\partial D,
\end{cases}\label{Model of homogenization equation}
\end{equation}
where 
\begin{align}\notag
\begin{cases}
\left(\nabla\cdot(C(\dfrac{x}{\epsilon})\nabla u)\right)_{j}=\sum_{i,k,\ell=1}^{d}\dfrac{\partial}{\partial x_{i}}\left(C_{ijk\ell}(\dfrac{x}{\epsilon})\dfrac{\partial u_{\ell}}{\partial x_{k}}\right)\text{ for }1\leq j\leq d,\\
\Delta^{*}=\mu\Delta+(\lambda+\mu)\nabla(\nabla \cdot)=\nabla \cdot (C^{\iss 0}\nabla \cdot ),
\end{cases}
\end{align}
and $\nu$ is the unit outward normal on $\partial D$, $u^{s}$ is the
scattered field, and $u^{in}$ is an incident field; the operator $T_{\nu}$ stands for the \textit{boundary
traction operator} of the isotropic elasticity system (from the exterior
domain), which is 
\begin{align} \label{boundary traction operator}
T_{\nu}u=\begin{cases}
2\mu\dfrac{\partial u}{\partial\nu}+\lambda\nu\nabla\cdot u+\mu\nu^{T}(\partial_{2}u_{1}-\partial_{1}u_{2}), & \mbox{ when }n=2,\\
2\mu\dfrac{\partial u}{\partial\nu}+\lambda\nu\nabla\cdot u+\mu\nu\times(\nabla\times u), & \mbox{ when }n=3.
\end{cases}
\end{align}
The equation $\Delta^{*}u^{s}+\omega^{2}u^{s}=0$ with constant coefficients is called the Navier's
equation. In particular we consider an incident field that is either a plane shear wave 
\[
u^{in}=u_{s}^{in}:=d^\perp\exp(i\omega_{s}x\cdot d),
\]
where $d, d^\perp$ are orthonormal vectors in $\mathbb{R}^{d}$, or a plane pressure wave 
\[
u^{in}=u_{p}^{in}:=d\exp(i\omega_{p}x\cdot d),
\]
where $\omega_s$ and $\omega_p$ are denoted by $\omega_{s}=\dfrac{\omega}{\sqrt{\mu}}$
and $\omega_{p}=\dfrac{\omega}{\sqrt{\lambda+2\mu}}$.

Via the well-known Helmoholtz decomposition in $\mathbb{R}^{d}\backslash\overline{D}$,
one can see that the scattered field can be decomposed as
\[
u^{s}=u_{p}^{sc}+u_{s}^{sc},
\]
with 
\[
u_{p}^{sc}=-\dfrac{1}{\omega_{p}^{2}}\nabla(\nabla \cdot u^{sc})\mbox{ and }u_{s}^{sc}=\dfrac{1}{\omega_{s}^{2}}\mbox{rot}(\mbox{rot}u^{sc}),
\]
where $\mbox{rot}=\nabla^{T}$ represents $\frac{\pi}{2}$ clockwise rotation of the gradient if  $d=2$ and $\mbox{rot}=\nabla\times$ stands for the curl operator if $d=3$. The vector functions $u_{p}^{sc}$ and $u_{s}^{sc}$ are
called the pressure (longitudinal) and shear (transversal) parts of
the scattered vector field $u^{s}$, respectively and they satisfy
the Helmholtz equation 
\[
(\Delta+\omega_{p}^{2})u_{p}^{sc}=0\mbox{ and }\mbox{rot}u_{p}^{sc}=0\mbox{ in }\mathbb{R}^{d}\backslash\overline{D},
\]
\[
(\Delta+\omega_{s}^{2})u_{s}^{sc}=0\mbox{ and }\nabla \cdot u_{s}^{sc}=0\mbox{ in }\mathbb{R}^{d}\backslash\overline{D}.
\]
Furthermore, for the elastic scattering problem, the scattered field $u^{s}$ satisfies the \textit{Kupradze radiation condition} 
\begin{equation}
\lim_{r\to\infty}\left(\dfrac{\partial u_{p}^{sc}}{\partial r}-i\omega_{p}u_{p}^{sc}\right)=0\mbox{ and }\lim_{r\to\infty}\left(\dfrac{\partial u_{s}^{sc}}{\partial r}-i\omega_{s}u_{s}^{sc}\right)=0,\quad r=|x|,\label{Kupradze radiation condition}
\end{equation}
uniformly in all directions $\widehat{x}=\dfrac{x}{|x|}$. 

In summary the elastic scattering by highly oscillating periodic media can be formulated as:  find the solution $u^{\epsilon}\in H^{1}_{loc}(\mathbb{R}^{d})$ to
\begin{equation}
\begin{cases}
\nabla\cdot\left(C(\dfrac{x}{\epsilon})\nabla u^{\epsilon}\right)+\omega^{2}\rho(\dfrac{x}{\epsilon})u^{\epsilon}=0 & \mbox{ in }D,\\
\Delta^{*}u^{\epsilon}+\omega^{2}u^{\epsilon}=0 & \mbox{ in }\mathbb{R}^{d}\backslash\overline{D},\\
(u^{\epsilon})^+-(u^{\epsilon})^-=f & \mbox{ on }\partial D,\\
(T_{\nu}u^{\epsilon})^{+}-\left(C(\dfrac{x}{\epsilon})\nabla u^{\epsilon}\right)^{-}\cdot\nu=g & \mbox{ on }\partial D,
\end{cases}\label{Homogenization Equation}
\end{equation}
where $u^{\epsilon}$ satisfies the Kupradze radiation condition \eqref{Kupradze radiation condition}
at infinity. Here $u^\epsilon$ stands for the solution parameterized by $\epsilon$ and 
\[
f:=-u^{in} \quad \mbox{ and } \quad g:=-T_{\nu}u^{in} \quad \mbox{ on } \quad \partial D,
\]
where $u^{in}$ is either a shear wave or a pressure wave given as before. The superscripts $"+"$ or $"-"$ stand for the limit from exterior
or interior on $\partial D$, respectively. We remark that the highly oscillating periodic media is only supported in $D$. 

\subsection{Main results and outline}
We are interested in the limit behavior or the overall behavior of the solution $u^\epsilon$ as $\epsilon\to0$, known as \textit{homogenization}.
As $\epsilon\to0$, we are expecting that $u^{\epsilon}\to u^{\scriptscriptstyle (0)}$,
where $u^{\scriptscriptstyle (0)}$ is the solution of the homogenized equation 
\begin{equation}
\begin{cases}
\nabla\cdot(\overline{C}\nabla u^{\scriptscriptstyle (0)})+\omega^{2}\overline{\rho}u^{\scriptscriptstyle (0)}=0 & \mbox{ in }D,\\
\Delta^{*}u^{\scriptscriptstyle(0)}+\omega^{2}u^{\scriptscriptstyle (0)}=0 & \mbox{ in }\mathbb{R}^{d}\backslash\overline{D},\\
(u^{\ssi (0)})^+-(u^{\scriptscriptstyle (0)})^-=f & \mbox{ on }\partial D,\\
(T_{\nu}u^{\scriptscriptstyle (0) })^{+}-(\overline{C}\nabla u^{\scriptscriptstyle (0)})^{-}\cdot\nu=g & \mbox{ on }\partial D,
\end{cases}\label{0- Homogenized Equation}
\end{equation}
where $\overline{C}=(\overline{C}_{ijk\ell})$ is the constant four tensor
and $\overline{\rho}$ is the constant density given by 
\[
\begin{cases}
\overline{C}_{ijk\ell}=\int_{Y}\left(C_{ijk\ell}-C_{ijmn}\dfrac{\partial}{\partial y_{m}} \chi_{nk\ell}\right)dy,\\
\overline{\rho}=\int_{Y}\rho\, dy,
\end{cases}
\]
and here we have utilized the Einstein summation convention for the repeated indices; the constant fourth-order tensor $\overline{C}$ is called the effective
fourth-order tensor; the third-order tensor $\chi_{nk\ell}\in H_{per}^{1}(Y)$
is uniquely determined by the cell problem
\begin{equation}
\begin{cases}
\dfrac{\partial}{\partial y_{i}}\left(C_{ijmn}-C_{ijk\ell}\dfrac{\partial}{\partial y_{k}} \chi_{\ell m n}\right)=0\mbox{ in } Y,\\
\int_{Y}\chi_{\ell mn}(y)\,dy=0,
\end{cases}\label{Cell problem}
\end{equation}
where we refer to Appendix Section \ref{Section 5} for detailed analysis.
$H_{per}^1(Y)$ denotes the periodic Sobolev space consists of $H^1$ functions defined on the $d$-dimensional torus $\R^d / \mathbb{Z}^d$ where $Y=[0,1]^d$, we refer readers to \cite[Chapter 3]{cioranescu2000introduction} for detailed characterizations. 

The constant tensor $\overline{C}$ also satisfies the strong convexity condition \eqref{strong convexity} (see Appendix Section \ref{Section 5}), which implies that \eqref{0- Homogenized Equation} is a well-posed transmission problem. Note that from the standard elliptic regularity theory (see \cite{mclean2000strongly} for instance),  the corrector $\chi$ is $C^\infty$-smooth due to the smoothness of $C$. 

Now let us  look for ansatz  $u^\epsilon$ 
\begin{align*}
u^\epsilon =u^{\ssi (0)}+\epsilon u^{\ssi(1)} +\epsilon ^2 u^{\ssi(2)} +\cdots
\end{align*} 
as an asymptotic expansion in terms of $\epsilon$, where the functions $u^{\ssi (j)}$ will be characterized in the following sections for $j=0,1,2,\cdots$.
Now, we can state our main results in this paper. 
\begin{enumerate}
\item[(a)] $H^1$ convergence and $L^2$ convergence: the first result is the convergent rates between solutions $u^{\epsilon}$ and $u^{\ssi (0)}$.

\begin{theorem}[Convergence in $L^2$ and $H^1$]\label{MAIN THEOREM}
Let $u^\epsilon$ and $u^{\iss 0}$ be the solutions of \eqref{Homogenization Equation} and \eqref{0- Homogenized Equation}, respectively. Let $u^{\iss 1}$ be the bulk corrector given by \eqref{bulk correction} in $D$ with $u^{\iss 1}=0$ in $\mathbb R^d\setminus \overline{D}$. Then for any ball $B_R$ with $D \subset B_R$, we have
\begin{equation}\label{L^2 rates convergence 1}
\|u^\epsilon -u^{\iss 0}-\epsilon u^{\iss 1} \|_{H^1(D)}+\|u^\epsilon - u^{\iss 0}\|_{H^1(B_R\setminus \overline{D})}\leq C_R\epsilon ^{1/2}\|u^{\iss 0}\|_{H^2(D)},
\end{equation}
and  
\begin{equation}\label{L^2 rates convergence}
	\|u^\epsilon -u^{\ssi (0)}\|_{L^2(B_R)}\leq C_R\epsilon \|u^{\ssi(0)}\|_{H^2(D)},
\end{equation}
for some constant $C_R>0$ independent of $\epsilon$.

\end{theorem}
Furthermore we have the following higher-order convergent rates between solutions $u^\epsilon$ and $u^{\ssi(0)}$.
\begin{theorem}[Higher-order convergence in $L^2$ and $H^1$]\label{MAIN THEOREM 2}
Let $u^\epsilon$ and $u^{\iss 0}$ be the solutions of \eqref{Homogenization Equation} and \eqref{0- Homogenized Equation} respectively. Let $u^{\iss 1}$ and $u^{\iss 2}$ 
be defined by equations \eqref{higher u_1 bulk} and \eqref{high u_2} in $D$, respectively,
with $u^{\iss 1}=0$ and $u^{\iss 2}=0$ 
in $\mathbb R^d\setminus \overline{D}$. Let $\varphi^\epsilon $ and $\theta^\epsilon$ be the boundary correctors given by \eqref{varphi_epsilon} and \eqref{boundary corrector 2nd order}. 
Then for any ball $B_R$ with $D \subset B_R$, we have
\begin{equation}\label{higher order error estimate 1}
\|u^\epsilon -(u^\iss 0+\epsilon u^\iss 1+\epsilon^2 u^\iss 2 + \epsilon \varphi^\epsilon + \epsilon^2 \theta^\epsilon ) \|_{H^1(B_R)}\leq C_R\epsilon^2 \|u^{\iss 0}\|_{H^4(D)},
\end{equation}
and  
\begin{equation}\label{higher order error estimate}
\|u^\epsilon -(u^{\iss 0}+\epsilon u^{\iss 1}+\epsilon \varphi^\epsilon) \|_{L^2(B_R)}\leq C_R\epsilon^2 \|u^{\iss 0} \|_{H^4(D)},
\end{equation}
where $C_R>0$ is a constant independent of $\epsilon$ and $u^\iss 0$.
\end{theorem}

Remark that the above convergence estimates include boundary correctors $\varphi^\epsilon$ and $\theta^\epsilon$ since the periodic media has bounded support, as contrary to the case that the periodic media occupies $\R^d$.
\item[(b)] A second-order homogenization and wave dispersion: the higher-order homogenization enables to study the anisotropic dispersion of wave propagation in periodic media.  Formally the averaged wave field $U$ of $u_\epsilon$ up to order $\epsilon^2$ is governed by the fourth-order equation in $D$   
\begin{align} \label{fourth-order PDE}
\notag  \nabla \cdot \big( \overline{C} \nabla U \big) + \omega^2 \overline{\rho} U &=  -\epsilon \big( F : \nabla^3 U + \omega^2 G: \nabla U \big) \\
&~- \epsilon^2 \big( D: \nabla^4 U +   \omega^2 E: \nabla^2 U \big)  + O(\epsilon^3),
\end{align}
where  $D$ is a sixth-order tensor, $E$ is a fourth-order tensor, $F$ is a fifth-order tensor and $G$ is a third-order tensor respectively (see Section \ref{Section 4} for the definitions and details). The fourth-order partial differential equation in $D$ formally introduces the dispersion as is seen from the right hand side. In the low-frequency long-wavelength regime for wave propagation in periodic media, the wave dispersion  has been demonstrated by a fourth-order partial differential equation for the acoustic case \cite{cakoni2016homogenization}. In the particular case that the periodic media occupies $\R^d$, \eqref{fourth-order PDE} models the wave propagation and transcends the quasi-static regime. To the authors' knowledge, our second-order homogenization for elastic wave is new in the literature.
\end{enumerate}

\medskip

This article is further structured as follows. In Section \ref{Section 2}, we study the asymptotic analysis for the elastic homogenization problem. From the analysis, we can prove the $L^2$ convergent rates for the elastic homogenization problem, which shows Theorem \ref{MAIN THEOREM}. In Section \ref{Section 3}, we develop the higher-order asymptotic analysis. This gives us more delicate information about the convergent rates between solutions $u^\epsilon$ and $u^{\iss 0}$ and we can use it to prove Theorem \ref{MAIN THEOREM 2}. We further study a second-order homogenized model for the elastic system in Section \ref{Section 4}, where the anisotropic dispersion was demonstrated. In Appendix, for self-contained proofs, we offer fundamental materials which is used to demonstrate our homogenization theory for the elastic scattering problem.

\subsection{Notation}
\begin{enumerate}
\item
We use sub-index to represent the component of a tensor, in particular the $i_1 i_2 \cdots i_n$ component of a $n$-th order tensor $\chi$ is represented by $\chi_{i_1 i_2 \cdots i_n}$.
\item We use Einstein summation convention for the repeated indices.
\item $\cdot$ denotes the standard inner product and $:$ denotes the standard contraction between two tensors.
\item $C_c^\infty(\Omega)$ denotes the space consists of $C^\infty$ functions that are compactly supported in $\Omega$.
\end{enumerate}

\subsection{Acknowledgment}
The work was initiated when the authors participated the annual program on ``Mathematics and Optics"
 (2017--2018) at the Institute for Mathematics and its Applications (IMA) at the University of Minnesota. Y.-H. Lin would like to thank the support from IMA for his stay at the University of Minnesota.

\section{Asymptotic analysis of the transmission problem}\label{Section 2}

To begin with, we recall the  two-scale homogenization method for the elasticity scattering in periodic media.
\subsection{Basic asymptotic analysis}
Let us  consider $x$ and $y=\dfrac{x}{\epsilon}$
that are the slow and fast variables, respectively. Let $u^{\epsilon}$
be a solution of \eqref{Homogenization Equation} and we rewrite \eqref{Homogenization Equation}
to a first-order system 
\begin{equation}
\begin{cases}
v^{\epsilon}-C(\dfrac{x}{\epsilon})\nabla u^{\epsilon}=0,\\
\nabla\cdot v^{\epsilon}+\omega^{2}\rho(\dfrac{x}{\epsilon})u^{\epsilon}=0,\textsl{}
\end{cases}\mbox{ for }x\in D.\label{System of differential equation}
\end{equation}
Note that $u^{\epsilon}=(u^{\epsilon}_{i})_{1\leq i\leq d}$ is a
vector-valued function and $v^{\epsilon}=(v^{\epsilon}_{ij})_{1\leq i,j \leq d}$ is
a matrix-valued function. The two-scale homogenization method begins with ansatz $u^{\epsilon}=u^{\epsilon}(x,y)$ and $v^{\epsilon}=v^{\epsilon}(x,y)$
such that  
\begin{eqnarray}
 & u^{\epsilon}(x,y)= & u^{\iss 0}(x,y)+\epsilon u^{\iss 1}(x,y)+\epsilon^{2}u^{\iss 2}(x,y)+\cdots=\sum_{k=0}^{\infty}\epsilon^{k}u^{\iss k}(x,y),\nonumber \\
 & v^{\epsilon}(x,y)= & v^{\iss 0}(x,y)+\epsilon v^{\iss 1}(x,y)+\epsilon^{2}v^{\iss 2}(x,y)+\cdots=\sum_{k=0}^{\infty}\epsilon^{k}v^{\iss k}(x,y).\label{asymptotic expansions}
\end{eqnarray}
Furthermore since there are no microstructure in the exterior domain $\mathbb{R}^{d}\backslash\overline{D}$,  the asymptotic expansions of
$u^{\epsilon}$ and $v^{\epsilon}$ are nothing but 
\[
u^{\epsilon}=u^{\iss 0}(x)\mbox{ and }v^{\epsilon}=v^{\iss 0}(x).
\]

Proceeding with the ansatz, we consider $x$ and $y=\dfrac{x}{\epsilon}$ as independent
variables and correspondingly
\begin{equation}
\nabla=\nabla_{x}+\dfrac{1}{\epsilon}\nabla_{y}. \label{Chain rule}
\end{equation}
Use the formal expansion and combine \eqref{System of differential equation},
\eqref{asymptotic expansions} and \eqref{Chain rule}, then we have
\begin{equation}
\begin{cases}
\sum_{m=0}^{\infty}\epsilon^{m}v^{\iss m}(x,y)-C(y)\Big(\nabla_{x}+\dfrac{1}{\epsilon}\nabla_{y}\Big)\Big(\sum_{m=0}^{\infty}\epsilon^{m}u^{\iss m}(x,y)\Big)=0,\\
\Big(\nabla_{x}+\dfrac{1}{\epsilon}\nabla_{y} \Big)\cdot\Big(\sum_{m=0}^{\infty}\epsilon^{m}v^{\iss m}(x,y)\Big)+\omega^{2}\rho(y)\sum_{m=0}^{\infty}\epsilon^{m}u^{\iss m}(x,y)=0.
\end{cases}\label{expaned equation}
\end{equation}
By collecting the $\epsilon^m$ terms in equation \eqref{expaned equation} with $m=-1,0,1\cdots$,   
\begin{align}
O(\epsilon^{-1}):\qquad & C\nabla_{y}u^{\iss 0}=0 ,\label{epsilon -1 u}\\
 & \nabla_{y}\cdot v^{\iss 0}=0.\label{epsilon -1 v}\\
O(1):\qquad  & v^{\iss 0}-C\left(\nabla_{x}u^{\iss 0}+\nabla_{y}u^{\iss 1}\right)=0, \label{epilson 0 u}\\
 & \left(\nabla_{x}\cdot v^{\iss 0}+\nabla_{y}\cdot v^{\iss 1}\right)+\omega^{2}\rho u^{\iss 0}=0.\label{epsilon 1 v}
\end{align}
Via \eqref{epsilon -1 u}, we get $u^{\iss 0}=u^{\iss 0}(x)$. From \eqref{epsilon -1 v},
\eqref{epilson 0 u} and the cell function $\chi_{\ell mn}$ solving \eqref{Cell problem},
we can find that the bulk corrector  $u^{\iss 1}=(u^{\iss 1}_{\ell})_{1\leq\ell\leq d}$ is given component-wisely by
\begin{equation}\label{bulk correction}
u^{\iss 1}_{\ell}(x,y)=-\chi_{\ell mn}(y)\dfrac{\partial u^{\iss 0}_{n}}{\partial x_{m}}(x).
\end{equation}
We call $u^{\iss1}(x,y)$ the  first-order corrector for the periodic homogenization problem. Plugging equation \eqref{bulk correction} directly into \eqref{epilson 0 u} yields
\begin{align}
\notag v^{\iss 0}_{ij}(x,y) & =  \left(C(y)\left(\nabla_{x}u^{\iss 0}(x)+\nabla_{y}u^{\iss 1}(x,y)\right)\right)_{ij}\\
 & =  C_{ijkl}(y)\dfrac{\partial u^{\iss 0}_{\ell}}{\partial x_{k}}-C_{ijkl}\dfrac{\partial\chi_{\ell m n }}{\partial y_{k}}\dfrac{\partial u^{\iss 0}_{n}}{\partial x_{m}}, \label{standard v_0}
\end{align}
and consequently the $Y$-average of $v_{0}$ is 
\[
\langle  v^{\iss 0} \rangle:= \overline{v}^{\iss 0}= \int_{Y}v^{\iss 0}(x,y)dy=\overline{C}\nabla u^{\iss 0}.
\]

In order to solve $v^{\iss 1}$, we introduce the following partial differential equation in the unit cell. Let $q(x,y)\in H_{per}^{1}(\mathbb{R}^{d\times d};Y)$ be a solution to
\begin{equation}
\label{rotyq} \mbox{rot}_{y}(q)=v^{\iss 0}-\overline{C}\nabla u^{\iss 0},
\end{equation}
where $q(x,y)$ is 
\begin{align}
\begin{cases}\notag
 q=(q_1,q_2)\text{ when }d=2\text{ and }q_i\text{'s are scalar functions for }i=1,2,\\
 q=(q_1,q_2,q_3)\text{ when }d=3\text{ and }q_i\text{'s are column vectors for }i=1,2,3,
\end{cases}
\end{align} 
and let $\gamma_{m\ell}(y)$ be an $Y$-periodic vector-valued function solving the following equation
\begin{align}\label{equation gamma}
\begin{cases}
\dfrac{\partial }{\partial y_i} \Big(  C_{ijk\ell} \dfrac{\partial \gamma_{m\ell}}{\partial y_k}\Big) = \big(\overline{\rho} -\rho \big) \delta_{jm} ,\\
\int_Y \gamma_{m\ell} (y)dy=0.
\end{cases}
\end{align}
We remark that the right hand side of \eqref{equation gamma} has zero mean (i.e., $\int _Y(\overline\rho -\rho )\,dy=0$), and then compatibility condition is satisfied, which means \eqref{equation gamma} is solvable.

From \eqref{rotyq}, \eqref{equation gamma}, \eqref{epsilon 1 v}  and note that $\nabla_y\cdot(\mathrm{rot}_xq(x,y))=-\nabla _x\cdot (\mathrm{rot}_yq(x,y))$, one candidate for $v^{\iss 1}$ is $\widetilde v^{\iss 1}=(\widetilde v^{\iss 1}_{ij})_{1\leq i,j\leq d}$ defined by 
\begin{equation}\label{equation tilde v1}
\widetilde{v}_{ij}^{\iss 1}(x,y):=\left(\text{rot}_x(q(x,y))\right)_{ij}+\omega ^2 C_{ijk\ell}(y)\dfrac{\partial \gamma_{m\ell}}{\partial y_k}(y) u_m^{\iss 0}(x),\text{ for }1\leq i,j\leq d.
\end{equation}
We remark that the function $\widetilde{v}^{\iss 1}$ in the form \eqref{equation tilde v1} is convenient in proving Theorem \ref{MAIN THEOREM}, and we will choose another form of $\widetilde{v}^{\iss 1}$ to derive higher-order estimate.

\subsection{Rates of convergence in $H^1$ and $L^2$}

Now let us introduce the boundary corrector   $\widetilde\varphi^\epsilon$ that solves
\begin{align}
\begin{cases}
\nabla \cdot \left(C(\dfrac{x}{\epsilon})\nabla \widetilde\varphi^\epsilon\right)+\omega ^2\rho(\dfrac{x}{\epsilon})\widetilde\varphi^\epsilon =0 &\mbox{ in }D,\\
\Delta^*\widetilde\varphi ^\epsilon +\omega ^2\widetilde\varphi^\epsilon =0 &\mbox{ in }\mathbb R^d\setminus \overline{D},\\
(\widetilde\varphi^\epsilon )^+ - (\widetilde\varphi^\epsilon )^- = u^{\iss 1} &\mbox{ on }\partial D,\\
(T_\nu \widetilde\varphi ^\epsilon )^+-(C(\dfrac{x}{\epsilon }))\nabla \widetilde\varphi ^\epsilon)^-\cdot \nu  =\left(\dfrac{v^{\iss 0}-\overline{v}^{\iss 0}}{\epsilon}+\widetilde{v}^{\iss 1}\right)\cdot \nu &\mbox{ on }\partial D,
\end{cases}\label{boundary corrector}
\end{align}
where $\widetilde\varphi^\epsilon$ satisfies the Kupradze radiation condition \eqref{Kupradze radiation condition}. By plugging \eqref{bulk correction} and \eqref{equation tilde v1} into \eqref{boundary corrector}, one can see that the transmission conditions on $\partial D$ is 
\begin{align}
\begin{cases}
\Big( \big(\widetilde\varphi^\epsilon\big) ^+-\big(\widetilde\varphi^\epsilon \big)^- \Big)_j = -\chi_{j mn}(y)\dfrac{\partial u^{\iss 0}_n}{\partial x_m},\    \mbox{ for }1\leq \alpha \leq d &\mbox{ on }\partial D\\
(T_\nu \widetilde\varphi ^\epsilon )^+-(C(\dfrac{x}{\epsilon }))\nabla \widetilde\varphi ^\epsilon)^-\cdot \nu  =\left[\text{rot}~q+\omega ^2 (C:\nabla \gamma)(y)u^{\iss 0}\right]\cdot \nu &\mbox{ on }\partial D,
\end{cases}\label{new transmission condition}
\end{align}
where $\text{rot} = \text{rot}_x + \frac{1}{\epsilon} \text{rot}_y$.

Before proceeding with the following lemma, we remark that the solution $u^{\iss 0}$ to \eqref{0- Homogenized Equation} is sufficiently smooth, since \eqref{0- Homogenized Equation} is a well-posed transmission problem and moreover the tensor $\overline{C}$ and density $\overline{\rho}$ are constants and the boundary data is sufficiently smooth.

\begin{lemma}\label{H1 convergence lemma}
Let $u^\epsilon$ and $u^{\iss 0}$ be the solutions of \eqref{Homogenization Equation} and \eqref{0- Homogenized Equation}, respectively. Let $u^{\iss 1}$ be the bulk corrector given by \eqref{bulk correction} in $D$ with $u^{\iss 1}=0$ in $\mathbb R^d\setminus \overline{D}$ and $\widetilde\varphi_\epsilon $ be the boundary corrector given by \eqref{boundary corrector}. Then for any ball $B_R$ with $D \subset B_R$, we have 
\begin{equation}\label{error estimate 1}
\|u^\epsilon -(u^{\iss 0}+\epsilon u^{\iss 1}+\epsilon \widetilde\varphi^\epsilon) \|_{H^1(B_R)}\leq C_R\epsilon \|u^{\iss 0} \|_{H^2(D)},
\end{equation}
where $C_R>0$ is a constant independent of $\epsilon$ and $u^{\iss 0}$.
\end{lemma}
\begin{proof}
Consider the error functions in $D$ given by 
\begin{equation}
\label{w_epsilon}w^\epsilon := u^\epsilon -u^{\iss 0} -\epsilon u^{\iss 1},
\end{equation}
and 
\begin{equation}
\label{zeta_epsilon}\zeta^\epsilon :=C(\frac{x}{\epsilon})\nabla u^\epsilon -v^{\iss 0} -\epsilon\widetilde{v}^{\iss 1},
\end{equation}
where $w^\epsilon$ is a vector-valued function and $\zeta^\epsilon$ is a matrix-valued function. From straightforward calculations, we can get 
\begin{align}
\begin{cases}
C(\dfrac{x}{\epsilon})\nabla w^\epsilon -\zeta^\epsilon=\epsilon \big( \widetilde{v}^{\iss 1}-C(y)\nabla _xu^{\iss 1} \big),\\
(\nabla \cdot \zeta^\epsilon)_j +\omega^2 \rho(y)w^\epsilon_j= -\epsilon \omega^2 \Big(  \rho(y)u^{\iss 1}_j+C_{ijk\ell}(y)\dfrac{\partial \gamma_{m\ell}}{\partial y_k}\dfrac{\partial u^{\iss 0}_m}{\partial x_i} \Big), \text{ for }1\leq j\leq d,
\end{cases}\label{system of DEs}
\end{align} 
which is a first order differential system of $(w^\epsilon,\zeta^\epsilon)$ such that their right hand sides are of order $O(\epsilon)$.

While outside $D$ we simply consider the error functions $w^\epsilon:=u^\epsilon-u^{\iss 0}$ and $\zeta ^\epsilon := C^{\iss 0} \nabla w^\epsilon$, and then they satisfy
\begin{equation*}
-\nabla \cdot \zeta^\epsilon = \omega ^2 w^\epsilon .
\end{equation*}
Let $B_R$ be a sufficiently large ball that contains $\overline{D}$ and $\phi \in C^\infty _c (B_R)$ be a vector-valued test function, we shall derive an estimate for
\begin{align}
\int _{B_R}(w^\epsilon - \epsilon \widetilde\varphi^\epsilon )\cdot\phi dx &. \label{lemma 1 quantity}
\end{align}

To estimate the above quantity we consider another auxiliary function $\Phi^\epsilon \in H^1_{loc} (\R^d)$ by 
\begin{align}\label{corrector with test function}
\begin{cases}
\nabla \cdot (C(\dfrac{x}{\epsilon})\nabla \Phi ^\epsilon)+\omega ^2\rho(\dfrac{x}{\epsilon})\Phi^\epsilon =\phi &\mbox{ in }D,\\
\Delta^*\Phi ^\epsilon +\omega ^2\Phi^\epsilon =\phi &\mbox{ in }\mathbb R^d\setminus \overline{D},\\
(\Phi^\epsilon)^+-(\Phi^\epsilon)^- = 0 &\mbox{ on }\partial D,\\
(T_\nu\Phi ^\epsilon )^+-(C(\dfrac{x}{\epsilon }))\nabla \Phi ^\epsilon)^-\cdot \nu  =0 &\mbox{ on }\partial D,
\end{cases}
\end{align}
where $\phi$ is the test function as we mentioned before and $\Phi^\epsilon $ further satisfies the Kupradze radiation condition \eqref{Kupradze radiation condition} at infinity. Now we replace $\phi$ in \eqref{lemma 1 quantity} by  \eqref{corrector with test function}

\begin{align}
\notag\int _{B_R}(w^\epsilon -\epsilon \widetilde\varphi^\epsilon)\cdot\phi dx=&\int _D(w^\epsilon-\epsilon\widetilde\varphi^\epsilon)\cdot\left(\nabla \cdot (C(\dfrac{x}{\epsilon})\nabla \Phi^\epsilon)+\omega^2 \rho(\dfrac{x}{\epsilon})\Phi^\epsilon\right)dx\\
&\notag+\int_{B_R\setminus D}(w^\epsilon-\epsilon \widetilde\varphi^\epsilon)\cdot(\Delta^*\Phi^\epsilon+\omega^2 \Phi^\epsilon)dx\\
=&\notag -\int_D \big( C(\frac{x}{\epsilon})\nabla w^\epsilon \big) : \nabla { \Phi^\epsilon} dx+\epsilon \int_D \big( C(\frac{x}{\epsilon})\nabla\widetilde\varphi^\epsilon \big):\nabla { \Phi^\epsilon}dx \\
&\notag +\int_D \omega ^2 \rho (\frac{x}{\epsilon})(w^\epsilon-\epsilon \widetilde\varphi^\epsilon)\cdot \Phi^\epsilon dx+\int_{\partial D}T_\nu(w^\epsilon -\epsilon \widetilde\varphi^\epsilon)^+ \cdot \Phi^\epsilon dS\\
&\label{some computation 1}+\int_{\partial B_R}(w^\epsilon-\epsilon\widetilde\varphi ^\epsilon)\cdot T_\nu \Phi^\epsilon dS-\int_{\partial B_R}T_\nu (w^\epsilon-\epsilon\widetilde \varphi^\epsilon)\cdot \Phi^\epsilon dS,
\end{align}
where we used the integration by parts formula once for the interior $D$ and twice for the exterior $B_R\setminus \overline{D}$, and the function $w^\epsilon -\epsilon\widetilde\varphi ^\epsilon $ has no jumps across $\partial D$. Furthermore the last two terms of \eqref{some computation 1} are zero, due to the Kupradze radiation condition; indeed from equations (\ref{DtN_Dirichlet}), (\ref{DtN_Neumann}) in Appendix \ref{Appd_DtN} on the discussion of Kupradze radiation condition, a direct calculation yields
\begin{align*}
&\int_{\partial B_R}(w^\epsilon-\epsilon\widetilde\varphi ^\epsilon)\cdot T_\nu  \Phi^\epsilon dS-\int_{\partial B_R}T_\nu (w^\epsilon-\epsilon \widetilde\varphi^\epsilon)\cdot \Phi^\epsilon dS=0.
\end{align*}
Now we have
\begin{align}
\int _{B_R}(w^\epsilon -\epsilon \widetilde\varphi^\epsilon)\cdot\phi dx
=&\notag -\int_D \big( C(\frac{x}{\epsilon})\nabla w^\epsilon \big): \nabla \Phi^\epsilon dx+\epsilon \int_D \big( C(\frac{x}{\epsilon})\nabla\widetilde\varphi^\epsilon \big):\nabla \Phi^\epsilon dx \\
&\notag +\int_D \omega ^2 \rho (\frac{x}{\epsilon})(w^\epsilon-\epsilon \widetilde\varphi^\epsilon)\cdot \Phi^\epsilon dx+\int_{\partial D}T_\nu(w^\epsilon -\epsilon \widetilde\varphi^\epsilon)^+\cdot  \Phi^\epsilon dS.
\end{align}
Since $\widetilde\varphi^\epsilon$ satisfies equation \eqref{boundary corrector}, then from integration by parts
\begin{align}\notag
\int _{B_R}(w^\epsilon -\epsilon \widetilde\varphi^\epsilon)\cdot\phi dx
=&-\int_D \big( C(\frac{x}{\epsilon})\nabla w^\epsilon\big) :\nabla \Phi^\epsilon dx+\epsilon\int _{\partial D}\big(C(\frac{x}{\epsilon})\nabla \widetilde\varphi ^\epsilon\big)^-\nu\cdot \Phi^\epsilon dx\\
&\notag +\omega ^2 \int_D\rho(\frac{x}{\epsilon})w^\epsilon\cdot \Phi^\epsilon dx+\int _{\partial D}\big(T_\nu (w^\epsilon -\epsilon \widetilde\varphi^\epsilon)\big)^+\cdot \Phi^\epsilon dS.
\end{align}
From  \eqref{boundary corrector}, \eqref{system of DEs} and integration by parts we can further obtain
\begin{align}\notag
&\int _{B_R}(w^\epsilon -\epsilon \widetilde\varphi^\epsilon)\cdot\phi dx\\
=&\notag -\int _D\zeta^\epsilon\cdot\nabla \Phi^\epsilon dx+\omega ^2 \int _D\rho (\frac{x}{\epsilon})w_\epsilon \cdot \Phi^\epsilon dx+\int _{\partial D}(T_\nu w^\epsilon)^+\cdot \Phi^\epsilon dS\\
&\notag+\epsilon\int _D \big(-\widetilde{v}^{\iss 1}+C(\frac{x}{\epsilon}) \nabla_x u^{\iss 1} \big) : \nabla \Phi^\epsilon dx+\int_{\partial D}(\overline{v}^{\iss 0} -v^{\iss 0}-\epsilon \widetilde{v}^{\iss 1})\nu\cdot \Phi^\epsilon dS\\
\notag=&-\epsilon\omega^2  \int_D\Big( \rho(\dfrac{x}{\epsilon})u^{\iss 1}_j+C_{ijk\ell}(\dfrac{x}{\epsilon})\dfrac{\partial \gamma_{m\ell}}{\partial y_k}\dfrac{\partial u^{\iss 0}_m}{\partial x_i}\Big) \Phi^\epsilon_j dx \\
&\notag+\epsilon \int _D \left(-\widetilde{v}^{\iss 1}+C(\frac{x}{\epsilon})\nabla_xu^{\iss 1}\right):\nabla \Phi^\epsilon dx\\
&\notag +\int_{\partial D}\Big((\overline{v}^{\iss 0}-v^{(0)}-\epsilon \widetilde{v}^{\iss 1})^-\cdot \nu)-(\zeta^\epsilon )^-\cdot \nu+(T_\nu w^\epsilon )^+ \Big)\cdot \Phi^\epsilon dS.
\end{align}
From equations \eqref{Homogenization Equation}, \eqref{0- Homogenized Equation},  \eqref{boundary corrector}, \eqref{w_epsilon} and \eqref{zeta_epsilon} we can obtain that the the last term in the above equality is zero. Thus, 
\begin{align}\label{CS applying}
&\int _{B_R}(w^\epsilon -\epsilon \widetilde\varphi^\epsilon)\cdot\phi dx\\
=&\notag-\epsilon\omega^2  \int_D\Big(\rho(\dfrac{x}{\epsilon})u^{\iss 1}_j+C_{ijk\ell}(\dfrac{x}{\epsilon})\dfrac{\partial \gamma_{m\ell}}{\partial y_k}\dfrac{\partial u^{\iss 0}_m}{\partial x_i} \Big) \Phi^\epsilon_j dx+\epsilon \int _D \Big(-\widetilde{v}^{\iss 1}+C(\frac{x}{\epsilon})\nabla_xu^{\iss 1}\Big):\nabla \Phi^\epsilon dx,
\end{align}
Note that the function $q$ solves \eqref{rotyq}, then one can choose $q$ such that 
\begin{align*}\label{tilde v1 estimate}
\sup_{y\in Y}|\widetilde{v}^{\iss 1}|\leq C\left(2\sum _{i,j}\left|\dfrac{\partial u^{\iss 0}}{\partial x_i \partial x_j}\right|+|u^{\iss 0}|\right),
\end{align*}
where $C>0$ is a constant independent of $\epsilon $. Recall that $u^{\iss 1}$ is represented by \eqref{bulk correction}, then we obtain
\begin{align}\notag
\left\|\rho u^{\iss 1}_j+C_{ijk\ell}\dfrac{\gamma_{m\ell}}{\partial y_k}\dfrac{\partial u^{\iss 0}_m}{\partial x_i}\right\|_{L^2(D)}\leq C\|u^{\iss 0}\|_{H^2(D)} \mbox{ and } \left\|C_{ijk\ell}\dfrac{\partial u^{\iss 1}_\ell}{\partial x_k}\right\|_{L^2(D)}\leq C\|u^{\iss 0} \|_{H^2(D)},
\end{align}
where the first inequality holds for $1 \leq j \leq d$ and the second inequality holds for $1\leq i,j \leq d$. Appling the Cauchy-Schwartz inequality to \eqref{CS applying} we can obtain
\begin{equation*}
\left|\int _{B_R}(w^\epsilon -\epsilon \widetilde\varphi^\epsilon )\cdot\phi dx\right|\leq C\epsilon \|u^{(0)} \|_{H^2(D)}\|\Phi^\epsilon\|_{H^1(D)},
\end{equation*}
for some constant $C>0$ independent of $\epsilon$. Finally from the standard estimate for the elliptic system (see \cite{mclean2000strongly} for instance), 
\begin{equation*}
\|\Phi^\epsilon \|_{H^1(D)}\leq C\|\phi \|_{H^{-1}(B_R)},
\end{equation*}
where $C>0$ is a constant depends on the coefficients and $R$. Finally the proof follows from the duality argument.
\end{proof}

Now we can have the following theorem.
\begin{theorem}
Let $u^\epsilon$ and $u^{\iss 0}$ be the solutions of \eqref{Homogenization Equation} and \eqref{0- Homogenized Equation}, respectively. Let $u^{\iss 1}$ be the bulk correction given by \eqref{bulk correction} in $D$ with $u^{\iss 1}=0$ in $\mathbb R^d\setminus \overline{D}$, then we have 
\begin{equation*}
\|u^\epsilon -u^{\iss 0}-\epsilon u^{\iss 1} \|_{H^1(D)}+\|u^\epsilon - u^{\iss 0}\|_{H^1(B_R\setminus \overline{D})}\leq C_R\epsilon ^{1/2}\|u^{\iss 0}\|_{H^2(D)},
\end{equation*}
for some constant $C_R>0$ independent of $\epsilon$.
\end{theorem}
\begin{proof}
From the elastic transmission problem \eqref{boundary corrector}, the function $\widetilde\varphi^\epsilon$ satisfies the following $H^1$ estimate (see Theorem \ref{estimate on transmission problem} in Appendix Section \ref{Section 5}),  
\begin{align}
&\notag \|\widetilde\varphi^\epsilon \|_{H^1(D)}+\|\widetilde\varphi^\epsilon\|_{H^1(B_R\setminus \overline{D})}\\
\leq &C_R \Big(\|u^{\iss 1}\|_{H^{1/2}(\partial D)}+\left\|\left(\dfrac{\overline{v}^{\iss 0}-v^{\iss 0}}{\epsilon}-\widetilde{v}^{\iss 1}\right)\cdot \nu \Big\|_{H^{-1/2}(\partial D)}\right),\label{H1 estimate}
\end{align}
for some constant $C_R>0$ independent of $\epsilon $. Recall that $u^{\iss 1}_\ell (x,\dfrac{x}{\epsilon})= \chi_{\ell mn}(\dfrac{x}{\epsilon})  \dfrac{\partial u^{\iss 0}_n}{\partial x_m}(x)$, by a standard argument of the trace theorem and cutoff techniques in the homogenization theorem (see \cite[Chapter 7]{cioranescu2000introduction} for instance), we have 
\begin{align}
\| u^{\iss 1}\|_{H^{1/2}(\partial D)}\leq C\epsilon ^{-1/2}\|u^{\iss 0} \|_{H^2(D)}, \label{u_1 estimate}
\end{align}
where $C>0$ is a constant independent of $\epsilon$. From  equations \eqref{boundary corrector} and \eqref{new transmission condition}, 
\begin{eqnarray*}
\left
( \dfrac{\overline{v}^{\iss 0}-v^{\iss 0}}{\epsilon}-\widetilde{v}^{\iss 1}\right)\cdot \nu  = \left(\text{rot}~q+\omega ^2 C(y)\nabla \gamma(y)u^{\iss 0}\right)\cdot \nu,
\end{eqnarray*}
where $O(\frac{1}{\epsilon})$ term is absorbed.

Now let $\phi$ be any arbitrary smooth vector-valued test function, then when $d=2$, we have 
\begin{align*}
\int_{\partial D}\text{rot}(q_j)\cdot \nu \phi\,dS=-\int _{\partial D}q_j\ \text{rot}(\phi)\cdot \nu \,dS,\text{ for }j=1,2,
\end{align*}
and when $d=3$, 
\begin{align*}
\int_{\partial D}\text{rot}(q_j)\cdot \nu \phi\, dS=-\int _{\partial D}(q_j\times \nabla \phi)\cdot \nu \,dS,\text{ for }j=1,2,3.
\end{align*}
From the governing equation \eqref{rotyq} of $q_j(y)$ for $1\leq j\leq d$, we can see that the $H^1_{per}(Y)$-norm of $q=(q_j)_{1\leq j\leq d}$ is bounded by $\|u^{\iss 0}\|_{H^2(D)}$. From the trace theorem we can obtain
\begin{eqnarray*}
\| q_j \|_{L^2(\partial D)} \le C \|u^{\iss 0}\|_{H^2(D)} \quad \mbox{and} \quad \| q_j \|_{H^1(\partial D)} \le C\epsilon^{-1} \|u^{\iss 0}\|_{H^2(D)},
\end{eqnarray*}
where $C>0$ is a constant independent of $\epsilon$. Then from the above inequalities and the duality argument, 
\begin{align}\label{H minus one boundary estimate}
\Big\|\Big(\dfrac{v^{\iss 0}-\overline{v}^{\iss 0}}{\epsilon}+\widetilde{v}^{\iss 1}\Big)\cdot \nu \Big\|_{H^{-1}(\partial D)}\leq C\|u^{\iss 0} \|_{H^2(D)},
\end{align}
and 
\begin{align}\label{L2 boundary estimate}
\Big\|\Big(\dfrac{v^{\iss 0}-\overline{v}^{\iss 0}}{\epsilon}+\widetilde{v}^{\iss 1}\Big)\cdot \nu \Big\|_{L^2(\partial D)}\leq C\epsilon ^{-1}\|u^{\iss 0}\|_{H^2(D)},
\end{align}
where $C>0$ is a constant independent of $\epsilon$. Therefore by interpolating between \eqref{H minus one boundary estimate} and  \eqref{L2 boundary estimate}, and combining with \eqref{H1 estimate},  \eqref{u_1 estimate}, we obtain
\begin{equation}\label{estimate for boundary corrector}
\|\widetilde\varphi^\epsilon \|_{H^1(D)}+\|\widetilde\varphi^\epsilon\|_{H^1(B_R\setminus \overline{D})}\leq C_R\epsilon ^{-1/2}\|u^{\iss 0}\|_{H^2(D)},
\end{equation}
for some constant $C_R>0$ independent of $\epsilon$. Finally from \eqref{error estimate 1} and \eqref{estimate for boundary corrector} we obtain
\begin{equation*}
\|u^\epsilon -u^{\iss 0}-\epsilon u^{\iss 1} \|_{H^1(D)}+\|u^\epsilon - u^{\iss 0}\|_{H^1(B_R\setminus \overline{D})}\leq C_R\epsilon ^{1/2}\|u^{\iss 0}\|_{H^2(D)},
\end{equation*}
where $C_R>0$ is some constant independent of $\epsilon $ and this completes the proof. This also proves \eqref{L^2 rates convergence 1} in Theorem \ref{MAIN THEOREM}.
\end{proof}

Consider another boundary corrector function as follows. Let $\varphi^\epsilon$ be a boundary corrector, which solves the following equation 
\begin{align}\label{varphi_epsilon}
\begin{cases}
\nabla \cdot \big(C(\dfrac{x}{\epsilon})\nabla \varphi^\epsilon\big)+\omega^2\rho(\dfrac{x}{\epsilon})\varphi^\epsilon =0 &\text{ in }D,\\
\Delta^* \varphi^\epsilon+\omega ^2 \varphi^\epsilon =0 &\text{ in }\mathbb R^d \setminus \overline{D},\\
(\varphi^\epsilon)^+-(\varphi^\epsilon)^-=u^{\iss 1} &\text{ on }\partial D,\\
(T_\nu \varphi^\epsilon)^+-\big(C(\dfrac{x}{\epsilon})\nabla \varphi^\epsilon \cdot \nu\big)^-=\big(\dfrac{v^{\iss 0}-\overline{v}^{\iss 0}}{\epsilon}+v^{\iss 1}\big)\cdot \nu &\text{ on }\partial D,
\end{cases}
\end{align} 
where $v^{\iss 1}$ is the function given in the asymptotic expansion \eqref{asymptotic expansions}. Here we remark that $v^{\iss 1}$ might not be the same function as $\widetilde{v}^{\iss 1}$.
\begin{lemma}
Let $B_R$ be an arbitrary ball in $\mathbb R^d$ such that $D\subset B_R$. Let $u^{\iss 0}\in H^2(B_R)$ be the solution of \eqref{0- Homogenized Equation} and $\varphi^\epsilon$ be the solution of \eqref{varphi_epsilon}, then we have 
\begin{equation*}\label{L^2 error bound}
\|\varphi^\epsilon \|_{L^2 (B_R)}\leq C_R\|u^{\iss 0} \|_{H^2(D)},
\end{equation*}
for some constant $C_R>0$ independent of $\epsilon$.

\end{lemma}
\begin{proof}
Let us consider a test function $\phi\in L^2(B_R)$ such that $\phi\equiv0$ outside $B_R$ and let $\Phi^\epsilon$ be the solution to the transmission problem \eqref{corrector with test function}. We begin with the estimate of $\widetilde \varphi ^\epsilon $. It is similar to the proof of Lemma \ref{H1 convergence lemma} and indeed we can obtain
\begin{align}\notag
\int _{B_R}\widetilde\varphi^\epsilon\cdot \phi dx=&-\int _D \big( C(\frac{x}{\epsilon})\nabla \widetilde\varphi^\epsilon \big):\nabla \Phi^\epsilon dx+\omega ^2 \int _D\rho (\frac{x}{\epsilon})\widetilde\varphi^\epsilon \cdot \Phi^\epsilon dx\\
&\notag +\int_{\partial D}\Big( \big(C(\frac{x}{\epsilon})\nabla \Phi^\epsilon \big)^-\cdot \nu\Big) \cdot (\widetilde\varphi^\epsilon)^- dS\\
&+\int _{\partial D}\Big(\big(\nabla \widetilde\varphi^\epsilon\big) ^+\cdot \nu\Big)\cdot (\Phi^\epsilon)^+dS-\int_{\partial D}\Big( \big(\nabla \Phi^\epsilon\big)^+\cdot \nu \Big)\cdot (\widetilde\varphi^\epsilon)^+ dS.
\end{align}
By using the integration by parts in $D$, the equation for $\widetilde\varphi^\epsilon$, Kupradze radiation condition \eqref{Kupradze radiation condition} for $\widetilde\varphi^\epsilon$ and continuous transmission boundary conditions for $\Phi^\epsilon$ (see \eqref{corrector with test function} again), we can derive
\begin{align}
\notag\int_{B_R}\widetilde\varphi^\epsilon \cdot \phi dx=&\int_{\partial D}\Big( \big(\dfrac{v^{\iss 0}-\overline{v}^{\iss 0}}{\epsilon}+\widetilde{v}^{\iss 1}\big)\cdot \nu\Big) \cdot (\Phi ^\epsilon)^+ dS-\int_{\partial D}u^{\iss 1}\cdot\Big( \big(\nabla \Phi^\epsilon\big)^+\cdot \nu \Big)dS\\
=&\notag \int_{\partial D}\bigg( \Big (\big(\text{rot} q \big)_{ij}+\omega ^2 C_{ijk\ell}(y)\dfrac{\partial \gamma_{m\ell}}{\partial y_k}(y) u^{\iss 0}_m \Big)\nu_i \bigg) \cdot (\Phi^\epsilon_j)^+dS\\
&+\int _{\partial D}\chi _{\ell mn}(\frac{x}{\epsilon}) \dfrac{\partial u^{\iss 0}_n }{\partial x_m}\Big( \big(\nabla \Phi^\epsilon\big)^+\cdot\nu\Big)_\ell \,dS.
\end{align}  

Let  $\Phi^{\iss 0}$ be the solution of the leading-order homogenized transmission problem with respect to $\Phi^\epsilon$,  and $\Phi^{\iss 1}_\ell (x,y)=\chi _{\ell mn} (y) \dfrac{\partial \Phi^{\iss 0}_n }{\partial x_m}(x)$ be the first order corrector term corresponding to $\Phi^\epsilon$. Furthermore let $\Psi^\epsilon$ be the bulk corrector of $\Phi^\epsilon$ as the role of  $\widetilde\varphi^\epsilon$ playing for $u^\epsilon$. 
Following the same proof of Lemma \ref{H1 convergence lemma}, we can derive that  
\begin{equation*}\label{Another H1 estimate}
\|\Phi^\epsilon -(\Phi^{\iss 0}+\epsilon\Phi^{\iss 1}+\epsilon \Psi^\epsilon)\|_{H^1(B_R)}\leq C_R\epsilon \|\Phi^{\iss 0}\|_{H^2(D)},
\end{equation*}
where  $C_R>0$ is a constant independent of $\epsilon$ and $\Phi^{\iss 0}$. Since the bulk correction of $\Phi^\epsilon$ is zero outside $\overline{D}$, then in particular we have 
\begin{equation*}
\|\Phi^\epsilon -(\Phi^{\iss 0}+\epsilon \Psi^\epsilon)\|_{H^1(B_R\setminus \overline{D})}\leq C_R\epsilon \|\Phi^{\iss 0} \|_{H^2(D)}.
\end{equation*}
From the definitions of $\Phi^\epsilon$, $\Phi^{\iss 0}$ and $\Psi^\epsilon$ in $B_R\setminus \overline{D}$, we know that $\nabla \cdot (C^{\iss 0}\nabla \Phi^\epsilon)$, $\nabla \cdot (C^{\iss 0}\nabla \Phi^{\iss 0})$ and $\nabla \cdot (C^{\iss 0} \nabla \Psi^\epsilon)$ belong to $L^2(B_R\setminus \overline{D})$, then from Appendix Section \ref{Section 5} 
\begin{equation}\label{estimate 1}
\Big\|\nabla\Big( \big(\Phi^\epsilon) ^+-(\Phi^{\iss 0} \big)^+-\epsilon (\Psi^\epsilon)^+\Big)\cdot \nu \Big\|_{H^{-1/2}(\partial D)}\leq C_R\epsilon \|\Phi^{\iss 0}\|_{H^2(D)}.
\end{equation}

Therefore we can now obtain
\begin{align}
&\notag \left| \int_{\partial D}\bigg( \Big (\big(\text{rot} q \big)_{ij}+\omega ^2 C_{ijk\ell}(y)\dfrac{\partial \gamma_{m\ell}}{\partial y_k}(y) u^{\iss 0}_m \Big)\nu_i \bigg) \cdot (\Phi^{\iss 0}_j)^+dS \right|\\
\notag\leq &\ C\|q\|_{H^{-1}(\partial D)}\|\nabla \Phi^{\iss 0} \|_{H^1(\partial D)}+\omega^2\left|\int_{\partial D} C_{ijk\ell}(y)\dfrac{\partial \gamma_{m\ell}}{\partial y_k}(y) u^{\iss 0}_m  \nu_i  ~ (\Phi^{\iss 0}_j)^+dS\right|\\
\leq &\label{estimate 2}\  C\|u^{\iss 0}\|_{H^2(D)}\|\Phi^{\iss 0}\|_{H^2(D)},
\end{align}
where we have used $\|q_j\|_{H^{-1}(\partial D)}\leq C\|u^{\iss 0}\|_{H^2(D)}$ with the constant $C>0$ independent of $\epsilon$ (by \eqref{rotyq}) and $C_{ijk\ell}\dfrac{\partial \gamma _{m\ell}}{\partial y_k}$ is bounded in $H^{1/2}(\partial D)$ independent of $\epsilon$ for all $1\leq i,j,m\leq d$ (see \eqref{equation gamma}). We also know that 
\begin{equation}\notag
\|u^{\iss 1}\|_{L^2(\partial D)}\leq C\|u^{\iss 0}\|_{H^2(D)},
\end{equation}
for some constant $C>0$ independent of $\epsilon$ and hence, 
\begin{align}
\notag\left|\int_{\partial D}\left((\nabla \Phi^{(0)})^+\cdot\nu \right)\cdot u^{\iss 1} dS\right|\leq &\ C\|u^{\iss 1}\|_{L^2(\partial D)}\|\nabla \Phi^{\iss 0}\|_{L^2(\partial D)}\\
\leq & \ C\|u^{\iss 0}\|_{H^2(D)}\|\Phi^{\iss 0}\|_{H^2(D)}.
\end{align}\label{estimate 3}
To proceed, we can use similar arguments as before to get 
\begin{equation}\notag
\|\Psi^\epsilon \|_{H^{1/2}(\partial D)}\leq C\epsilon ^{-1/2}\|\Phi^{\iss 0}\|_{H^2(D)},
\end{equation}
and 
\begin{equation}\notag
\left\|\dfrac{v^{\iss 0}-\overline{v}^{\iss 0}}{\epsilon}+\widetilde v^{\iss 1}\right\|_{H^{-1/2}(\partial D)}\leq C\epsilon ^{-1/2}\|u^{\iss 0}\|_{H^2(D)},
\end{equation}
which implies that 
\begin{equation}
\left|\int_{\partial D}\left(\dfrac{v^{\iss 0}-\overline{v}^{\iss 0}}{\epsilon}+\widetilde{v}^{\iss 1}\right)\cdot \epsilon\Psi_\epsilon dS \right|\leq C\|u^{\iss 0}\|_{H^2(D)}\|\Phi^{\iss 0}\|_{H^2(D)},
\end{equation}\label{estimate 4}
where $C>0$ is a constant independent of $\epsilon$.
Similar arguments give 
\begin{equation}
\epsilon\left|\int_{\partial D}u^{\iss 1}\cdot((\nabla \Psi^\epsilon)^+ \cdot\nu )dS \right|\leq C\|u^{\iss 0}\|_{H^2(D)}\|\Phi^{\iss 0}\|_{H^2(D)},
\end{equation}\label{estimate 5}
for some constant $C>0$ independent of $\epsilon$, where we have utilized the fact that  
\begin{equation}\notag
\|u^{\iss 1}\|_{H^{1/2}(\partial D)}=\left\|\chi(\frac{x}{\epsilon})\nabla u^{\iss 0}\right\|_{H^{1/2}(\partial D)}\leq C\epsilon^{-1/2}\|u^{\iss 0}\|_{H^2(D)}.
\end{equation}

Finally from \eqref{estimate for boundary corrector} we can obtain that 
\begin{align}
&\notag\epsilon\left|- \int_D\left(C(\frac{x}{\epsilon})\nabla\widetilde \varphi^\epsilon\right):\nabla \Psi^\epsilon dx+\omega ^2 \int_D\rho(\dfrac{x}{\epsilon})\widetilde\varphi^\epsilon\cdot \Psi^\epsilon dx+\int_{\partial D}\left((C(\dfrac{x}{\epsilon})(\nabla \Psi^\epsilon)^-)\cdot \nu \right)\cdot\widetilde\varphi^\epsilon dS \right. \\
&\notag\left.+\int_{\partial D}((\nabla \widetilde\varphi^\epsilon)^+\cdot\nu )\cdot (\Psi^\epsilon)^+dS-\int_{\partial D}((\nabla \Psi^\epsilon)^+\cdot \nu)\cdot (\widetilde\varphi^\epsilon)^+ dS \right|\\
\leq\  &\label{estimate 6} C\epsilon \|\varphi^\epsilon\|_{H^1(B_R)}\|\Psi^\epsilon\|_{H^1(B_R)}\ \leq\ C\|u^{\iss 0}\|_{H^2(D)}\|\Phi^{\iss 0}\|_{H^2(D)},
\end{align}
for some constants $C>0$ independent of $\epsilon$, $u^{\iss 0}$ and $\Phi^{\iss 0}$. Then by combining \eqref{estimate 2}-\eqref{estimate 6} and the remainder term of \eqref{estimate 1} if of order $\epsilon$ 
\begin{equation*}
\|\widetilde \varphi^\epsilon \|_{L^2(B_R)} \|\phi\|_{L^2(B_R)} \leq C\|u^{\iss 0} \|_{H^2(D)} \|\Phi^{\iss 0}\|_{H^2(D)}.
\end{equation*}
Furthermore since $\|\Phi^{\iss 0}\|_{H^2(D)}\leq C\|\phi\|_{L^2(B_R)}$, then there exists a constant $C>0$ independent of $\epsilon$ such that 
\begin{equation}\label{L^2 bound of the original boundary corrector}
\|\widetilde \varphi^\epsilon \|_{L^2(B_R)}\leq C\|u^{\iss 0} \|_{H^2(D)}.
\end{equation}
Finally since the difference between $\widetilde \varphi^\epsilon$ and $\varphi^\epsilon$ only appears in the jump conormal derivative across the boundary $\partial D$ (see equations \eqref{boundary corrector} and \eqref{varphi_epsilon}), this proves the theorem.
\end{proof}
Now, we are ready to prove the rates of convergence of $\|u^\epsilon-u^{\iss 0}\|_{L^2(B_R)}$.

\begin{proof}[Proof of Theorem \ref{MAIN THEOREM}]
It is easy to see that 
\begin{equation*}
\label{bound of u1}\|u^{\iss 1}\|_{L^2(B_R)}= \|u^{\iss 1}\|_{L^2(D)}\leq C\|u^{\iss 0}\|_{H^2(D)},
\end{equation*}
by using the definition of $u^{\iss 1}$ and the smoothness of $\chi(y)$. From \eqref{error estimate 1} and \eqref{L^2 bound of the original boundary corrector}, one can see that 
\begin{align}
\notag\|u^\epsilon -u^{\iss 0}\|_{L^2(B_R)}&\leq C_R\epsilon \|u^{\iss 0} \|_{H^2(D)}+\epsilon\| u^{\iss 1}\|_{L^2(B_R)}+\epsilon\| \widetilde\varphi^\epsilon\|_{L^2(B_R)}\\
&\notag\leq C_R\epsilon \|u^{\iss 0} \|_{H^2(D)},
\end{align}
for some constant $C_R>0$ independent of $\epsilon$. This completes the proof.
\end{proof}

\section{Higher-order asymptotic analysis of the transmission problem}\label{Section 3}
There are recent interests on higher-order two-scale homogenization of wave propagation in periodic meida  \cite{allaire2016comparison, cakoni2016homogenization, chen2001dispersive, meng2017dynamic, wautier2015second}. In the case that the periodic structure was only supported in a bounded domain, contrary to the case that the periodic structure occupies $\R^d$, the boundary correctors played a role both in the leading-order and second-order homogenization as demonstrated in \cite{cakoni2016homogenization} for scalar wave equation. In this section we study the higher-order homogenization of the elastic scattering problem where the periodic media has bounded support.  
\subsection{Higher-order asymptotic expansion}

Recall in asymptotic expansion (\ref{expaned equation}) the first order term $u^{\iss 1}$ was given by (\ref{bulk correction}), in this section we consider a more general form of $u^{\iss 1}=u^{\iss 1}(x,y)$ given by 
\begin{align}
u^{\iss 1}_{\ell}(x,y) = -\chi_{\ell mn}(y) \frac{\partial u^{\iss 0}_n}{\partial {x_m}}(x) + \widetilde{u}^{\iss 1}_{\ell}(x).\label{higher u_1 bulk}
\end{align}
From the ansatz we further obtain
\begin{eqnarray}
O(\epsilon) : && \label{high asymp 1} v^{\iss 1} - C(y) \big( \nabla_x u^{\iss 1}+  \nabla_y u^{\iss 2} \big) = 0,  \\
   &&\label{high asymp 2} \big( \nabla_x \cdot v^{\iss 1} + \nabla_y \cdot v^{\iss 2}) +  \omega^2 \rho(y) u^{\iss 1} = 0 .
\end{eqnarray}

Now we first derive a representation for $u_2$. Applying divergence $\nabla_y \cdot$ to equation \eqref{high asymp 1} and using \eqref{epsilon 1 v} yield 
\begin{align}
\nabla _y \cdot \big( C(y) \nabla_y u^{\iss 2} \big)+ \nabla_y \cdot \big(  C(y) \nabla_x u^{\iss 1} \big) = \nabla_y \cdot  v^{\iss 1} =-\nabla_x\cdot v^{\iss 0} -\omega ^2 \rho(y)u^{\iss 0}. \label{higher order asym 1}
\end{align}
From equations \eqref{higher u_1 bulk}, \eqref{higher order asym 1}, \eqref{standard v_0} and direct computations, we obtain the governing equation for $u^{\iss 2}$ 
\begin{align}
\frac{\partial }{\partial y_i} \Big( C_{ijk\ell}(y) \frac{ \partial u^{\iss 2}_\ell}{\partial y_k} \Big) =& \Big( -C_{ijk\ell} + C_{ij mn} \frac{\partial \chi_{nk\ell}}{\partial y_m} + \frac{\partial }{\partial y_m} \big(  \chi_{ni\ell} C_{mjkn} \big)  + \overline{C}_{ijk\ell} \Big) \frac{\partial^2 u^{\iss 0}_{\ell}}{\partial x_k \partial x_i} \nonumber \\
& -\frac{C_{ijk\ell}}{\partial y_i} \frac{\partial \widetilde{u}^{\iss 1}_{\ell}}{\partial x_k} + \omega^2 (\overline{\rho} - \rho) u^{\iss 0}_j. \label{eqn u_2 derive}
\end{align}
Let us set 
\begin{eqnarray*}
b_{i j k \ell} =-C_{ijk\ell} + C_{ij mn} \frac{\partial \chi_{nk\ell}}{\partial y_m} + \frac{\partial }{\partial y_m} \big(  \chi_{ni\ell} C_{mjkn} \big) ,
\end{eqnarray*}
and note that  $\int_Y{b}_{ij k \ell}\,dy = -\overline{C}_{ijk\ell}$. Besides, due to the symmetric properties of $C_{ijk\ell}$, we know that $\overline{b}_{ijk\ell}$ also has the major and minor symmetry.

Now we introduce higher-order cell functions $\chi_{ik \ell q}$ that is $Y$-periodic function and solves  
\begin{align}\label{higher order corrector}
\frac{\partial }{\partial y_{\alpha}} \left( C_{\alpha j \beta q }(y)  \frac{\partial \chi_{ i k \ell q } }{\partial y_{\beta}} \right)  = b_{ijk\ell} - \int_Y{b}_{ij k \ell}\,dy.
\end{align}
In addition with the help of the cell functions $\chi_{\ell m n}$ defined by \eqref{Cell problem} and  $ \gamma_{m\ell}$ defined by \eqref{equation gamma}, one can directly obtain from equation \eqref{eqn u_2 derive} that
\begin{align}
u^{\iss 2}_{p}= \chi_{ mnqp}  \frac{\partial^2 u^{\iss 0}_{q}}{\partial x_n \partial x_m} - \chi_{pmn} \frac{ \partial \widetilde{u}^{\iss 1}_{n} }{\partial x_m}+ \omega^2 \gamma_{mp}(y) u^{\iss 0}_m+\widetilde{u}^{\iss 2}_p(x),\text{ for }1\leq p\leq d, \label{high u_2}
\end{align}
where the function $\widetilde{u}^{\iss 2}_p$ will be determined later. It is not hard to see that $u^{\iss 2}$ is a solution of \eqref{high asymp 1} (due to $\nabla_y\widetilde{u}^{\iss 2}(x)=0$). From \eqref{high asymp 1}
\begin{align}
v^{\iss 1} =  C(y) \big( \nabla_x u^{\iss 1} +  \nabla_y u^{\iss 2} \big), \label{high v_1}
\end{align}
then applying the divergence $\nabla_x \cdot$ to \eqref{high v_1} and note that $u^{\iss 1}$ and $u^{\iss 2}$ are given by \eqref{higher u_1 bulk} and \eqref{high u_2} respectively,
\begin{align}
(\notag \nabla_x \cdot v^{\iss 1})_j =&\left( -C_{ijn\ell} \chi_{\ell m q} + C_{ijk\ell} \frac{\partial \chi_{mnq \ell}}{\partial y_k} \right) \frac{\partial^3  {u}^{\iss 0}_{q}}{\partial x_i \partial x_m \partial x_n}\\
&+ \omega ^2 C_{mjk\ell} \frac{\partial \gamma_{n \ell}}{\partial y_k}\dfrac{\partial u^{\iss 0}_n}{\partial x_m}  +\left(C_{ijk\ell }-C_{ijmn}\dfrac{\partial\chi_{nk\ell}}{\partial y_m}\right)\dfrac{\partial^2  \widetilde{u}^{\iss 1}_\ell}{\partial x_i \partial x_k }. \label{nabla_x v_1}
\end{align}
Applying the divergence $\nabla_x\cdot$ to \eqref{high v_1} and then averaging that over $Y$ yield 
\begin{eqnarray*}
\int_Y \nabla_x \cdot v^{\iss 1} dy - \int_Y \nabla_x \cdot \Big(  C(y) \big( \nabla_x u^{\iss 1} +  \nabla_y u^{\iss 2} \big) \Big) dy   =0,  
\end{eqnarray*}
note that $u^{\iss 1}$, $u^{\iss 2}$ and $\nabla_x \cdot v^{\iss 1}$ are given by \eqref{higher u_1 bulk}, \eqref{high u_2} and \eqref{nabla_x v_1} respectively, then a direct calculation yields
\begin{eqnarray}
 \overline{C}_{ijk\ell}  \frac{\partial^2 \widetilde{u}^{\iss 1}_{\ell}}{\partial x_i \partial x_k} + \omega^2 \overline{\rho} \widetilde{u}^{\iss 1}_j \
&=&\notag -\left( \frac{\partial^3  {u}^{\iss 0}_{q}}{\partial x_i \partial x_m \partial x_n}\right) \int_Y \left( -C_{ijn\ell} \chi_{\ell m q} + C_{ijk\ell} \frac{\partial \chi_{mnq \ell}}{\partial y_k}   \right)dy  \\
&&   -\omega^2 \frac{\partial {u}^{\iss 0}_{n}}{\partial x_m}  \int_Y\left( -\rho \chi_{jmn} + C_{mjk\ell} \dfrac{\partial \gamma_{n \ell} }{\partial y_k}  \right)dy. \label{homogenized tilde{u}1} 
\end{eqnarray}

We will show that the function $\widetilde u^{\iss 1}(x)$ in $u^{\iss 1}(x,y)$ cannot be chosen as zero in the elastic homogenization case, which is different from the scalar case \cite{cakoni2016homogenization} (the function $\widetilde u_1$ can be taken by zero in the scalar case). 
Via integration by parts and periodic conditions of $C_{ijk\ell}$, $\chi_{mnq\ell}$ and the cell problem \eqref{Cell problem}, one can see that 
\begin{align*}
\notag &\int _Y C_{ijk\ell} \frac{\partial }{\partial y_k} \chi_{mnq \ell}dy=-\int_Y \chi _{mnq\ell }\dfrac{\partial }{\partial y_k}C_{k\ell ij}dy\\
=&-\int_Y\chi _{mnq\ell} \dfrac{\partial }{\partial y_k}\left(C_{k\ell \alpha \beta }\dfrac{\partial }{\partial y_\alpha}\chi _{\beta ij}\right)dy=-\int_Y \chi _{\beta ij}  \dfrac{\partial }{\partial y_\alpha}\left(C_{k\ell \alpha \beta }\dfrac{\partial }{\partial y_k}\chi _{mnq\ell}\right)dy,
\end{align*}
where we have used integration by parts twice in the last equality. From the symmetric condition of the fourth-order tensor $C_{k\ell \alpha \beta}=C_{\alpha \beta k\ell}$ and equation \eqref{higher order corrector}, we can get 
\begin{align}
\notag&\int _Y C_{ijk\ell} \frac{\partial }{\partial y_k} \chi_{mnq \ell}dy=-\int_Y \chi _{\beta ij}  \dfrac{\partial }{\partial y_\alpha}\left(C_{\alpha \beta k\ell}\dfrac{\partial }{\partial y_k}\chi _{mnq\ell}\right)dy\\
=\notag &-\int_Y\chi_{\beta ij}(b_{m\beta n q}-\int_Y b_{m\beta n q}\,dy)dy\\
=\notag&-\int_Y\chi_{\beta ij} \left(-C_{m\beta nq}+C_{m\beta \alpha \gamma }\dfrac{\partial \chi _{\gamma nq}}{\partial y_\alpha}+\dfrac{\partial}{\partial y_\alpha}(\chi_{\gamma mq}C_{\alpha \beta n\gamma})+\overline{C}_{m\beta nq}\right)dy\\
=&\int_Y\chi _{\beta ij}C_{m\beta nq}dy-\int _Y\chi _{\beta ij}C_{m\beta \alpha \gamma} \dfrac{\partial \chi_{\gamma nq}}{\partial y_\alpha}dy+\int_Y \chi_{\gamma mq}C_{\alpha \beta n\gamma}\dfrac{\partial \chi _{\beta ij}}{\partial y_\alpha}dy,\label{tedious calculation 1}
\end{align}
where we used the integration by parts and $\int_Y \chi _{\beta ij}dy=0$ in the last equality. Therefore from \eqref{tedious calculation 1} we can obtain
\begin{align}
\notag&\left( \frac{\partial^3  {u}^{\iss 0}_{q}}{\partial x_i \partial x_m \partial x_n}\right)\left(\int _Y C_{ijk\ell} \frac{\partial }{\partial y_k} \chi_{mnq \ell}dy-\int_YC_{ijn\ell}\chi _{\ell mq}dy\right)\\
=&\notag\left( \frac{\partial^3  {u}^{\iss 0}_{q}}{\partial x_i \partial x_m \partial x_n}\right)
\int_Y\left(-\chi _{\ell mq}C_{ijn\ell}+\chi _{\beta ij}C_{m\beta nq}-\chi _{\beta ij}C_{m\beta \alpha \gamma} \dfrac{\partial \chi_{\gamma nq}}{\partial y_\alpha}+ \chi_{\gamma mq}C_{\alpha \beta n\gamma}\dfrac{\partial \chi _{\beta ij}}{\partial y_\alpha}\right)dy.
\end{align}
From the above representation, it is easy to see that the above quantity may not be zero, since  the index $q$ induces non-symmetry among the indices $q,i,m,n$ even though the indices $i,m,n$ can be interchanged freely.

For the second term in the right hand side of \eqref{homogenized tilde{u}1}, from equation \eqref{equation gamma} governing $\gamma$ and  integration by parts, we have 
\begin{align*}
&\int_Y C_{mjk\ell} \frac{\partial}{\partial y_k} \gamma_{n\ell}  dy = - \int_Y  \gamma_{n\ell} \frac{\partial}{\partial y_k} C_{mjk\ell}  dy = - \int_Y \gamma_{n\ell}  \frac{\partial}{\partial y_k} C_{k\ell m j}  dy \\
=& -\int_Y \frac{\partial}{\partial y_k} \big(  C_{k\ell p q} \frac{\partial}{\partial y_p} \chi_{qmj} \big) \gamma_{n\ell} dy = -\int_Y \frac{\partial}{\partial y_p} \big( C_{k\ell p q}  \frac{\partial}{\partial y_k} \gamma_{n\ell}  \big) \chi_{qmj} dy \\
=& -\int_Y \frac{\partial}{\partial y_p} \big( C_{p q k \ell}  \frac{\partial}{\partial y_k} \gamma_{n\ell}  \big) \chi_{qmj} dy = -\int_Y \chi_{qmj} (\overline{\rho} - \rho) \delta_{qn} dy = \int_Y \rho \chi_{nmj} dy,
\end{align*}
whereby
\begin{align}
\notag &  \int_Y\left( -\rho \chi_{jmn} + C_{mjk\ell} \frac{\partial }{\partial y_k} \gamma_{n \ell} \right)dy = \int_Y\left( -\rho \chi_{jmn} + \rho \chi_{nmj}  \right)dy.
\end{align}
The fact that $\chi_{jmn}$ has symmetries with respect to $m$ and $n$ may not yield the above quantity to be zero. Note that for the scalar case $d=1$ (see  \cite{cakoni2016homogenization}), the right hand side of \eqref{homogenized tilde{u}1} is zero, thus one can choose $\widetilde u^{\iss 1}=0$ without loss of generality in the scalar case, but for the elastic case, we simply keep $\widetilde u^{\iss 1}(x)$ in the following analysis.

Now let us seek for the a formula for $v^{\iss 2}$ and we denote such a function by $\widehat{v}^{\iss 2}$ in this section. In particular from equation \eqref{high asymp 2}
\begin{eqnarray*}
\nabla_y \cdot v^{\iss 2}= -\nabla_x \cdot v^{\iss 1} -  \omega^2 \rho(y) u^{\iss 1}.
\end{eqnarray*}
From equations  \eqref{higher u_1 bulk}, \eqref{high u_2} and \eqref{high v_1}, one can derive the equation for $v^\iss 2$
\begin{align}
(\nabla_y \cdot  v^{\iss 2})_j  =& \Big( C_{ijn\ell} \chi_{\ell m q} -C_{ij k\ell} \frac{\partial \chi_{mnq\ell}}{\partial y_k} \Big)   \frac{\partial^3 u^{\iss 0}_{q}}{\partial x_i \partial x_m \partial x_n}  + \omega^2  \Big( -C_{mjk\ell}  \frac{\partial \gamma_{n\ell}}{\partial y_k} + \rho \chi_{jmn}  \Big) \frac{\partial u^{\iss 0}_{n}}{\partial x_m} \nonumber \\
& + \Big( -C_{ijkq} + C_{ijmn}   \frac{\partial \chi_{nkq}}{\partial y_m} \Big) \frac{\partial^2 \widetilde{u}^{\iss 1}_{q}}{\partial x_k \partial x_i} -  \omega^2 \rho \widetilde{u}^{\iss 1}_j. \nonumber
\end{align}
From the governing equation \eqref{homogenized tilde{u}1} for $\tilde{u}_1$, one can further simplify the above equation to
\begin{align}
(\nabla_y \cdot  v^{\iss 2})_j=& \Big( 
  C_{ijn\ell} \chi_{\ell m q} -C_{ij k\ell} \frac{\partial \chi_{mnq\ell}}{\partial y_k}- \int_Y \big( C_{ijn\ell} \chi_{\ell m q} -C_{ij k\ell} \frac{\partial \chi_{mnq\ell}}{\partial y_k} \big) dy
 \Big) \frac{\partial^3 u^{\iss 0}_{q}}{\partial x_i \partial x_m \partial x_n} \nonumber\\
 & + \omega^2 \Big( -C_{mjk\ell}  \frac{\partial \gamma_{n\ell}}{\partial y_k} + \rho \chi_{jmn}  - \int_Y\big(-C_{mjk\ell}  \frac{\partial \gamma_{n\ell}}{\partial y_k} + \rho \chi_{jmn} \big)\,dy 
 \Big) \frac{\partial u^{\iss 0}_n}{\partial x_m} \nonumber \\
& + \Big( 
 -C_{ijkq} + C_{ijmn}   \frac{\partial \chi_{nkq}}{\partial y_m}  - \int_Y\big({-C_{ijkq} + C_{ijmn}   \frac{\partial \chi_{nkq}}{\partial y_m}}\big)\,dy 
 \Big) \frac{\partial^2 \widetilde{u}^{\iss 1}_{q}}{\partial x_k \partial x_i} \nonumber\\
 &  - (\rho -\overline{\rho}) \omega^2 \widetilde{u}^{\iss 1}_j. \nonumber
\end{align}
Now  we introduce the following higher-order cell function $\widehat{\chi}_{inmq\ell}$, $\widehat{\gamma}_{i k q \ell}$ and $\widehat{\gamma}_{\ell m n}$ that are $Y$-periodic functions and solve
\begin{eqnarray}\label{higher order corrector 5-th order}
\begin{cases}
\dfrac{\partial }{\partial y_{\alpha}} \left( C_{\alpha j \beta \ell }(y)  \dfrac{\partial \widehat{\chi}_{ i nmq\ell } }{\partial y_{\beta}} \right)  = d_{ijnmq} - \int_Y{d}_{ijnmq}\,dy, \\
\dfrac{\partial }{\partial y_{\alpha}} \left( C_{\alpha j \beta \ell }(y)  \dfrac{\partial \widehat{\gamma}_{ i k q \ell } }{\partial y_{\beta}} \right)  = -C_{ijkq} + C_{ijmn}   \dfrac{\partial \chi_{nkq}}{\partial y_m} - \int_Y\left({-C_{ijk\ell} + C_{ijmn}   \dfrac{\partial \chi_{nkq}}{\partial y_m}}\right)\,dy, \\
\dfrac{\partial }{\partial y_{\alpha}} \left( C_{\alpha j \beta \ell }(y)  \dfrac{\partial \widehat{\gamma}_{ \ell m n } }{\partial y_{\beta}} \right)  =  -C_{mjk\ell}  \dfrac{\partial \gamma_{n\ell}}{\partial y_k} + \rho \chi_{jmn} - \int_Y\left({-C_{mjk\ell}  \dfrac{\partial \gamma_{n\ell}}{\partial y_k} + \rho \chi_{jmn}}\right)\,dy,
\end{cases}
\end{eqnarray}
where $d_{ijnmq}$ is defined by
\begin{eqnarray}\label{dijnmq}
d_{i j nmq} =C_{ijn\ell} \chi_{\ell m q} - C_{ij k\ell} \frac{\partial \chi_{mnq\ell}}{\partial y_k} .
\end{eqnarray}
Now let us define $\widehat{v}^{\iss 2}$ where the $\alpha \beta$-th component is given by
\begin{align} \label{high v_2}
\widehat{v}^{\iss 2}_{\alpha \beta} 
=& C_{\alpha \beta k \ell} \Big( \frac{\partial \widehat{\chi}_{inmq\ell}}{\partial y_k} \frac{\partial^3  {u}^{\iss 0}_{q}}{\partial x_i \partial x_m \partial x_n} +  \frac{\partial \widehat{\gamma}_{i p q \ell}}{\partial y_k} \frac{\partial^2  \widetilde{u}^{\iss 1}_{q}}{\partial x_i \partial x_p} + \omega^2 \frac{\partial \widehat{\gamma}_{\ell m n}}{\partial y_k} \frac{\partial  {u}^{\iss 0}_{n}}{\partial x_m} + \omega^2 \frac{\partial \gamma_{m\ell}}{\partial y_k}  \widetilde{u}^{\iss 1}_m \Big).
\end{align}
Then from equation (\ref{homogenized tilde{u}1}), \eqref{higher order corrector 5-th order} and \eqref{dijnmq}, one can directly verify that  $\widehat v^{\iss 2}$ satisfies equation (\ref{high asymp 2}).

Now let us introduce the boundary corrector function $ \theta^\epsilon$ that solves
\begin{align}
\begin{cases}
\nabla \cdot \left(C(\dfrac{x}{\epsilon})\nabla \theta^\epsilon\right)+\omega ^2\rho(\dfrac{x}{\epsilon})\theta^\epsilon =0 &\mbox{ in }D,\\
\Delta^*\theta ^\epsilon +\omega ^2\theta^\epsilon =0 &\mbox{ in }\mathbb R^d\setminus \overline{D},\\
(\theta^\epsilon)^+-(\theta^\epsilon)^- = u^{\iss 2} &\mbox{ on }\partial D,\\
(T_\nu \theta^\epsilon )^+-(C(\dfrac{x}{\epsilon }))\nabla\theta^\epsilon)^-\cdot \nu  =\widehat{v}^{\iss 2} \cdot \nu &\mbox{ on }\partial D,
\end{cases}\label{boundary corrector 2nd order}
\end{align}
where $\theta^\epsilon$ satisfies the Kupradze radiation condition \eqref{Kupradze radiation condition}.

\subsection{Rates of convergence in $L^2$ and $H^1$: The higher-order case}
Via previous discussions on higher-order asymptotic analysis, we have 

\begin{theorem}\label{H1 convergence lemma 2nd order}
Let $u^\epsilon$ and $u^{\iss 0}$ be the solutions of \eqref{Homogenization Equation} and \eqref{0- Homogenized Equation} respectively. Let $u^{\iss 1}$ and $u^{\iss 2}$  
be defined by equations \eqref{higher u_1 bulk} and \eqref{high u_2}, respectively, 
with $u^{\iss 1}=0$ and $u^{\iss 2}=0$ 
in $\mathbb R^d\setminus \overline{D}$. Let $\varphi^\epsilon $ and $\theta^\epsilon$ be the boundary correctors given by \eqref{varphi_epsilon} and \eqref{boundary corrector 2nd order}.  
Then for any ball $B_R$ with $D \subset B_R$, we have 
\begin{equation*}
\|u^\epsilon -(u^\iss 0+\epsilon u^\iss 1+\epsilon^2 u^\iss 2 + \epsilon \varphi^\epsilon + \epsilon^2 \theta^\epsilon ) \|_{H^1(B_R)}\leq C_R\epsilon^2 \|u^{\iss 0}\|_{H^4(D)},
\end{equation*}
where $C_R>0$ is a constant independent of $\epsilon$ and $u^\iss 0$.
\end{theorem}
\begin{proof}
The proof is similar to the proof of Theorem \ref{H1 convergence lemma}. Again consider error functions in $D$ defined by 
\begin{equation*}
\label{wepsilon 2nd order}w^\epsilon := u^\epsilon -u^\iss 0 -\epsilon u^\iss 1 - \epsilon^2 u^\iss 2,
\end{equation*}
and 
\begin{equation*}
\label{zetaepsilon 2nd order}\zeta^\epsilon :=C(\frac{x}{\epsilon})\nabla u^\epsilon -v^\iss 0 -\epsilon v^\iss 1 - \epsilon^2 \widehat{v}^\iss 2,
\end{equation*}
where $v^{\iss 1}$, $\widehat{v}^{\iss 2}$ are defined by \eqref{high v_1}, \eqref{high v_2} with $v^{\iss 1}=0$ and $\widehat{v}^{\iss 2}=0$ 
in $\mathbb R^d\setminus \overline{D}$, and $w^\epsilon$ is a vector-valued function and $\zeta^\epsilon$ is a matrix-valued function. In this proof we conveniently use the same notations as in the proof of Theorem \ref{H1 convergence lemma}, since it is clear from the context. From straightforward computations, we can get 
\begin{align}
\begin{cases}
C(\dfrac{x}{\epsilon})\nabla w^\epsilon -\zeta^\epsilon=\epsilon^2 (v^\iss 2-C(y)\nabla _x u^\iss 2),\\
\nabla \cdot \zeta^\epsilon +\omega^2 \rho(y)w^\epsilon = -\epsilon^2 \big[ \omega^2 \rho u^\iss 2 + \nabla_x \cdot \widehat{v}^\iss 2 \big],
\end{cases}\label{system of DEs 2nd order}
\end{align} 
and moreover
\begin{eqnarray}\label{system of DEs div 2nd order}
\nabla \cdot \zeta^\epsilon + \omega^2 \rho w^\epsilon = -\epsilon^2 \big(  \nabla_x \cdot \widehat{v}^\iss 2 + \omega^2 \rho u^\iss 2  \big).
\end{eqnarray}

Outside $D$ we simply define the error functions by $w^\epsilon:=u^\epsilon-u^\iss 0$ and $\zeta ^\epsilon :=\nabla w^\epsilon$, this directly gives 
\begin{equation*}
-\nabla \cdot \zeta^\iss\epsilon = \omega ^2 w^\epsilon.
\end{equation*}
Let $\phi \in C^\infty _c (B_R)$ be a vector-valued test function and consider an auxiliary function $\Phi^\epsilon$ that solves
\begin{align}\label{corrector with test function 2nd order}
\begin{cases}
\nabla \cdot \left(C(\dfrac{x}{\epsilon})\nabla \Phi ^\epsilon\right)+\omega ^2\rho(\dfrac{x}{\epsilon})\Phi^\epsilon =\phi &\mbox{ in }D,\\
\Delta^*\Phi ^\epsilon +\omega ^2\Phi^\epsilon =\phi &\mbox{ in }\mathbb R^d\setminus \overline{D},\\
(\Phi^\epsilon) ^+-(\Phi^\epsilon) ^- = 0 &\mbox{ on }\partial D,\\
(T_\nu\Phi ^\epsilon )^+-(C(\dfrac{x}{\epsilon }))\nabla \Phi ^\epsilon)^-\cdot \nu  =0 &\mbox{ on }\partial D,
\end{cases}
\end{align}
where $\Phi^\epsilon $ satisfies the Kupradze radiation condition \eqref{Kupradze radiation condition}. Thus from the same argument as in the proof of Theorem \ref{H1 convergence lemma}, one can get  

\begin{align}
&\notag\int _{B_R}(w^\epsilon -\epsilon  \varphi^\epsilon - \epsilon^2 \theta^\epsilon)\cdot\phi dx\\
=&\notag\int _D(w^\epsilon-\epsilon  \varphi^\epsilon - \epsilon^2 \theta^\epsilon)\cdot\left(\nabla \cdot (C(\dfrac{x}{\epsilon})\nabla \Phi^\epsilon)+\omega^2 \rho(\dfrac{x}{\epsilon})\Phi^\epsilon\right)dx\\
&\notag+\int_{B_R\setminus D}(w^\epsilon -\epsilon  \varphi^\epsilon - \epsilon^2 \theta^\epsilon)\cdot(\Delta^*\Phi^\epsilon+\omega^2 \Phi^\epsilon)dx\\
=&\notag -\int_D \big( C(\frac{x}{\epsilon})\nabla w^\epsilon\big):\nabla { \Phi^\epsilon} dx+\epsilon \int_D \big( C(\frac{x}{\epsilon})\nabla \varphi^\epsilon \big):\nabla { \Phi^\epsilon}dx \\
&\notag +  \epsilon^2 \int_D \big( C(\frac{x}{\epsilon})\nabla \theta^\epsilon \big):\nabla { \Phi^\epsilon}dx + \int_D \omega ^2 \rho (\frac{x}{\epsilon})(w^\epsilon -\epsilon  \varphi^\epsilon - \epsilon^2 \theta^\epsilon)\cdot \Phi^\epsilon dx\\
&+ \notag \int_{\partial D}T_\nu(w^\epsilon -\epsilon  \varphi^\epsilon - \epsilon^2 \theta^\epsilon)^+ \cdot \Phi^\epsilon dS.
\end{align}
From integration by parts, one can further obtain
\begin{align}\notag
&\int _{B_R}(w^\epsilon -\epsilon  \varphi^\epsilon - \epsilon^2 \theta^\epsilon)\cdot\phi dx\\
=&\notag -\int_D \big( C(\frac{x}{\epsilon})\nabla w^\epsilon \big) : \nabla { \Phi^\epsilon} dx+\epsilon \int_{\partial D}\big
( C(\frac{x}{\epsilon})(\nabla \varphi^\epsilon)^{-} \cdot v \big) \cdot { \Phi^\epsilon}dS+ \int_D \omega ^2 \rho (\frac{x}{\epsilon}) w^\epsilon \cdot \Phi^\epsilon dx \\
&\notag+ \epsilon^2 \int_{\partial D}\big( C(\frac{x}{\epsilon})(\nabla \theta^\epsilon)^{-} \cdot v \big) \cdot { \Phi^\epsilon}dS+\int_{\partial D}T_\nu(w^\epsilon -\epsilon  \varphi^\epsilon - \epsilon^2 \theta^\epsilon)^+ \cdot \Phi^\epsilon dS.
\end{align}
Then from equations \eqref{varphi_epsilon}, \eqref{boundary corrector 2nd order}, \eqref{system of DEs 2nd order} and \eqref{system of DEs div 2nd order}, and integration by parts one can obtain
\begin{align}
&\notag \int _{B_R}(w^\epsilon -\epsilon  \varphi^\epsilon - \epsilon^2 \theta^\epsilon)\cdot\phi dx\\
=&\notag -\int _D\zeta^\epsilon : \nabla \Phi^\epsilon dx+\omega ^2 \int _D\rho (\frac{x}{\epsilon})w^\epsilon \cdot \Phi^\epsilon dx+ \int_{\partial D} T_\nu w^{\epsilon +} \cdot \Phi^\epsilon dS\\
&\notag - \epsilon \int_{\partial D} \left( (T_\nu \varphi^\epsilon)^+ -  C(\frac{x}{\epsilon})(\nabla \varphi^\epsilon)^{-} \cdot v  \right)\cdot \Phi^\epsilon dS - \epsilon^2 \int_{\partial D} \left( (T_\nu \theta^\epsilon)^+ -  C(\frac{x}{\epsilon})(\nabla \theta^\epsilon)^{-} \cdot v  \right)\cdot \Phi^\epsilon dS   \\
& \notag - \epsilon^2 \int _{ D}\big( \widehat{v}^{\iss 2} - C(\frac{x}{\epsilon})\nabla_x u^{\iss 2}  \big) : \nabla \Phi^\epsilon dx  \\
\label{estimate err 2nd order}=&- \epsilon^2 \int _{ D}\big( \widehat{v}^{\iss 2} - C(\frac{x}{\epsilon}) \nabla_x u^{\iss 2}  \big) : \nabla \Phi^\epsilon dx -\epsilon^2 \int_D \big(  \nabla_x \cdot \widehat{v}^\iss 2 + \omega^2 \rho u^\iss 2  \big) dx.
\end{align}
From the equations of $u^{\iss 1}$, $u^\iss 2$, $v^\iss 1$ and $\widehat v^\iss 2$  defined by \eqref{higher u_1 bulk}, \eqref{high u_2}, \eqref{high v_1} and \eqref{high v_2}, one can obtain 
\begin{align}\notag
\Big \|\widehat v^\iss 2 - C(\frac{x}{\epsilon})\nabla_x u^\iss 2\Big \|_{L^2(D)}\leq C\|u^\iss 0\|_{H^4(D)} \mbox{ and } \left\|\nabla_x \cdot \widehat{v}^\iss 2 + \omega^2 \rho u^\iss 2  \right\|_{L^2(D)}\leq C\|u^\iss 0 \|_{H^4(D)},
\end{align}
where $C$ is a constant. Furthermore apply the Cauchy-Schwartz inequality on \eqref{estimate err 2nd order}, then we obtain 
\begin{equation*}
\left|\int _{B_R}(w^\epsilon -\epsilon  \varphi^\epsilon - \epsilon^2 \theta^\epsilon )\cdot\phi dx\right|\leq C\epsilon^2 \|u^\iss 0 \|_{H^4(D)}\|\Phi^\epsilon\|_{H^1(D)},
\end{equation*}
for some constant $C>0$ independent of $\epsilon$. Again we utilize the standard estimate for the elliptic system \eqref{corrector with test function 2nd order} (see \cite{mclean2000strongly} for instance), then we can obtain
\begin{equation*}
\|\Phi^\epsilon \|_{H^1(D)}\leq C\|\phi \|_{H^{-1}(B_R)},
\end{equation*}
where $C>0$ is a constant depends on the coefficients and $R$, but independent of $\epsilon$. By the duality arguments in the Sobolev space, then we complete the proof. This proves \eqref{higher order error estimate 1} in Theorem \ref{MAIN THEOREM 2}.
\end{proof}

Recall that  $u^\iss 1=0$, $u^\iss 2=0$, $v^\iss 1=0$ and $\widehat{v}^\iss 2=0$ in $\mathbb R^d\setminus \overline{D}$, without loss of generality, we can choose $\widetilde u^\iss 1$ solves the constant coefficient elliptic system \eqref{homogenized tilde{u}1} in $D$ with $\widetilde u^\iss 1=0$ on $\partial D$. Therefore, $\widetilde u^\iss 1$ is a smooth solution in $D$, by using the elliptic estimate of \eqref{homogenized tilde{u}1} again, then we obtain $\|\widetilde u^\iss 1\|_{L^2(B_R)}=\|\widetilde u^\iss 1\|_{L^2(D)}\leq C\|u^\iss 0\|_{H^4(D)}$ for some constant $C>0$. 

Now, we are ready to prove Theorem \ref{MAIN THEOREM 2}.
\begin{proof}[Proof of Theorem \ref{MAIN THEOREM 2}]
By using the same reason and arguments as before, one can easily see that 
\begin{align}\notag
\|u^\iss 2\|_{L^2(B_R)}\leq C_R\|u^\iss 0\|_{H^4(D)}\quad\text{and}\quad\|\theta^\epsilon \|_{L^2(B_R)}\leq C_R\|u^\iss 0\|_{H^4(D)},
\end{align}
for some constants $C_R>0$ independent of $\epsilon $. Therefore, the proof is nothing but a straightforward corollary by combining previous lemmas. Therefore, we prove \eqref{higher order error estimate} and complete our proof of Theorem \ref{MAIN THEOREM 2}.
\end{proof}
\section{A second-order homogenization and wave dispersion}\label{Section 4}
Dispersive models for wave propagation in periodic media transcends the usual quasi-static regime and is of great interest. A dispersive model for scalar wave equation was derived using Floquet-Bloch theory and higher-order asymptotic of the Bloch variety \cite{santosa1991dispersive}.  Alternatively higher-order two-scale homogenization enables to demonstrate the dispersive effective of wave propagation in periodic meida  \cite{allaire2016comparison, cakoni2016homogenization, chen2001dispersive, meng2017dynamic, wautier2015second}. The higher-order homogenization in particular sheds light on  sensing the microstructure through dispersion \cite{lambert2015bridging}. In this section we study the higher-order effective wave equation that can demonstrate dispersion of wave propagation in periodic media.   To begin with let us recall from asymptotic expansion (\ref{expaned equation}), we get
\begin{eqnarray}
O(\epsilon^2) : && \label{high asymp 1 2nd order} v^\iss 2 - C(y) \big( \nabla_x u^\iss 2 +  \nabla_y u^\iss 3 \big) = 0,  \\
   &&\label{high asymp 2 2nd order} \big( \nabla_x \cdot v^\iss 2 + \nabla_y \cdot v^\iss 3) +  \omega^2 \rho(y) u^\iss 2 = 0 .
\end{eqnarray}
Here we seek for a $v^\iss 2$ different  from equation \eqref{high v_2}. This is to be realized by finding $u^\iss 3$ first. Indeed taking the divergence of equation \eqref{high asymp 1 2nd order} respect to the $y$ variable and noting \eqref{high asymp 2},
\begin{align*}
\nabla _y \cdot \big( C(y) \nabla_y u^\iss 3 \big)+ \nabla_y \cdot \big(  C(y) \nabla_x u^\iss 2 \big) = \nabla_y \cdot  v^\iss 2 =-\nabla_x\cdot v^\iss 1 -\omega ^2 \rho(y)u^\iss 1. 
\end{align*}
With the help of \eqref{higher u_1 bulk}, \eqref{high u_2} and \eqref{high v_1}, 
A direct calculation yields
\begin{align}
\frac{\partial }{\partial y_\alpha} \left( C_{\alpha j \beta \ell}(y) \frac{ \partial u^\iss 3_\ell}{\partial y_\beta} \right) =& \Big( 
-\frac{\partial }{\partial y_p} \big( C_{pji\ell} \chi_{mnq\ell} \big)+ \big( C_{ijn\ell} \chi_{\ell m q} -C_{ij k\ell} \frac{\partial \chi_{mnq\ell}}{\partial y_k} \big) \nonumber\\
&- \int_Y \big( C_{ijn\ell} \chi_{\ell m q} -C_{ij k\ell} \frac{\partial \chi_{mnq\ell}}{\partial y_k} \big) dy
 \Big) \frac{\partial^3 u^\iss 0_q}{\partial x_i \partial x_m \partial x_n} \nonumber\\
& + \Big( 
\frac{\partial }{\partial y_m} \big( C_{mjkn} \chi_{niq} \big)+ \big( -C_{ijkq} + C_{ijmn}   \frac{\partial \chi_{nkq}}{\partial y_m} \big) \nonumber\\
&- \int_Y\big({-C_{ijkq} + C_{ijmn}   \frac{\partial \chi_{nkq}}{\partial y_m}}\big)\,dy 
 \Big) \frac{\partial^2 \widetilde{u}^\iss 1_{q}}{\partial x_k \partial x_i} \nonumber\\
 & + \omega^2 \Big( 
-\frac{\partial }{\partial y_p} \big( C_{pjm\ell} \gamma_{n\ell} \big)+ \big( -C_{mjk\ell}  \frac{\partial \gamma_{n\ell}}{\partial y_k} + \rho \chi_{jmn}  \big) \nonumber\\
&- \int_Y\big(-C_{mjk\ell}  \frac{\partial \gamma_{n\ell}}{\partial y_k} + \rho \chi_{jmn} \big)\,dy 
 \Big) \frac{\partial u^\iss 0_{n}}{\partial x_m} \nonumber \\
 & -  \frac{\partial C_{ijkq}}{\partial y_i} \frac{\partial \widetilde{u}^\iss 2_q}{\partial x_k} - (\rho -\overline{\rho}) \omega^2 \widetilde{u}^\iss 1_j \label{eqn u3 dispersive}
\end{align}
Let us introduce the higher-order cell functions $\chi_{inmq\ell}$ and $\gamma_{\ell m n}$ that are $Y$-periodic and solve
\begin{eqnarray} 
\frac{\partial }{\partial y_{\alpha}} \left( C_{\alpha j \beta \ell }(y)  \frac{\partial \chi_{ i nmq\ell } }{\partial y_{\beta}} \right)  &=& -\frac{\partial }{\partial y_p} \big( C_{pji\ell} \chi_{mnq\ell} \big) + \big( C_{ijn\ell} \chi_{\ell m q} -C_{ij k\ell} \frac{\partial \chi_{mnq\ell}}{\partial y_k} \big) \nonumber \\
&&- \int_Y\big( C_{ijn\ell} \chi_{\ell m q} -C_{ij k\ell} \frac{\partial \chi_{mnq\ell}}{\partial y_k} \big) dy \label{chi 5-th dispersive}\\
\frac{\partial }{\partial y_{\alpha}} \left( C_{\alpha j \beta \ell }(y)  \frac{\partial \gamma_{ \ell m n } }{\partial y_{\beta}} \right)  &=&  -\frac{\partial }{\partial y_p} \big( C_{pjm\ell} \gamma_{n\ell} \big) + \big( -C_{mjk\ell}  \frac{\partial \gamma_{n\ell}}{\partial y_k} + \rho \chi_{jmn} \big) \nonumber \\
&& - \int_Y\big({-C_{mjk\ell}  \frac{\partial \gamma_{n\ell}}{\partial y_k} + \rho \chi_{jmn}}\big)\,dy.\label{gamma 3-rd dispersive}
\end{eqnarray}
By changing the index one can see that the governing equation \eqref{higher order corrector} for $\chi_{ikq\ell}$ can be written as
\begin{eqnarray} 
\frac{\partial }{\partial y_{\alpha}} \left( C_{\alpha j \beta \ell }(y)  \frac{\partial \chi_{ i k q \ell} }{\partial y_{\beta}} \right)  &=& \frac{\partial }{\partial y_m} \big( C_{mjkn} \chi_{niq} \big) + \big(-C_{ijkq} + C_{ijmn}   \frac{\partial \chi_{nkq}}{\partial y_m} \big) \nonumber \\
&& - \int_Y\big({-C_{ijkq} + C_{ijmn}   \frac{\partial \chi_{nkq}}{\partial y_m}}\big)\,dy. \label{chi 4-th dispersive}
\end{eqnarray}
From \eqref{eqn u3 dispersive}, \eqref{chi 5-th dispersive}, \eqref{gamma 3-rd dispersive} and \eqref{chi 4-th dispersive}, one can obtain that  
\begin{align}
u^\iss 3_{\ell}= \chi_{ inmq\ell}  \frac{\partial^3 u^\iss 0_{q}}{\partial x_i \partial x_m \partial x_n} +  \chi_{ik q \ell} \frac{ \partial^2 \widetilde{u}^\iss 1_{q} }{\partial x_i \partial x_k}+ \omega^2 \gamma_{\ell mn}(y) \frac{\partial u^\iss 0_n}{\partial x_m}-  \chi_{\ell k q} \frac{\partial \widetilde{u}^\iss 2_q}{\partial x_k} + \omega^2 \gamma_{m\ell} \widetilde{u}^\iss 1_m, \label{high u_3}
\end{align}
for $1\leq \ell \leq d$.

Now we have from equation \eqref{high asymp 1 2nd order} that 
\begin{eqnarray*} \label{high v_2 new}
v^\iss 2 = C(y) \big( \nabla_x u^\iss 2 +  \nabla_y u^\iss 3 \big),
\end{eqnarray*}
then from the representation of $u^\iss 2$ and $u^\iss 3$ in equations  \eqref{high u_2} and  \eqref{high u_3} respectively, one can obtain
\begin{align}
&\notag(v^\iss 2)_{\alpha \beta} 
= C_{\alpha \beta k \ell} \big(  \frac{\partial \chi_{inmq\ell}}{\partial y_k} + \chi_{mnq\ell} \delta_{ki} \big)  \frac{\partial^3 u^\iss 0_{q}}{\partial x_i \partial x_m \partial x_n} + C_{\alpha \beta k \ell} \big(  \frac{\partial \chi_{imq\ell}}{\partial y_k} - \chi_{\ell mq} \delta_{ki} \big)  \frac{\partial^2 \widetilde{u}^\iss 1_{q}}{\partial x_i \partial x_m} \nonumber \\
&+\omega^2 C_{\alpha \beta k \ell} \big(  \frac{\partial \gamma_{\ell mq}}{\partial y_k} + \gamma_{q \ell} \delta_{mk} \big)  \frac{\partial {u}^\iss 0_{q}}{\partial x_m} + C_{\alpha \beta k \ell} \big(  -\frac{\partial \chi_{\ell mq}}{\partial y_k} + \delta_{q \ell} \delta_{mk} \big)  \frac{\partial \widetilde{u}^\iss 2_{q}}{\partial x_m} + \omega^2 C_{\alpha \beta k \ell}\frac{\partial \gamma_{q\ell}}{\partial y_k} \widetilde{u}^\iss 1_q. \label{high order v_2 new componentwise}
\end{align}
Taking the $Y$-average of equation \eqref{high asymp 2 2nd order} yields
\begin{eqnarray*}
\int_Y \nabla_x \cdot v^\iss 2 dy   +  \omega^2 \int_Y \rho u^\iss 2 dy = 0,
\end{eqnarray*}
note further that $v^{\iss 2}$ and $u^{\iss 2}$ are given by \eqref{high order v_2 new componentwise} and \eqref{high u_2} respectively,  then from a direct calculation one can obtain the following equation for $\widetilde{u}^\iss 2$ 
\begin{align}
\overline{C}_{\alpha \beta m q}  \frac{\partial^2 \widetilde{u}^\iss 2_{q}}{\partial x_\alpha \partial x_m} + \omega^2 \overline{\rho} \widetilde{u}^\iss 2_\beta
 =&\notag  -\frac{\partial^4 u^\iss 0_{q}}{\partial x_\alpha \partial x_i \partial x_m \partial x_n}  \int_Y \big( C_{\alpha \beta k \ell}\frac{\partial \chi_{inmq\ell}}{\partial y_k}  + C_{\alpha \beta i \ell} \chi_{mnq\ell} \big) ~dy\\
& - ~\omega^2 \frac{\partial^2 u^\iss 0_{q}}{\partial x_\alpha  \partial x_m} \Big(  \int_Y \big( C_{\alpha \beta k \ell}\frac{\partial \gamma_{\ell m q}}{\partial y_k}  + C_{\alpha \beta m \ell} \gamma_{q\ell} \big) ~dy \nonumber \\
& \qquad \qquad \qquad \quad + \int_Y \rho \chi_{m\alpha q \beta} ~dy -(\overline{\rho})^{-1} \overline{C}_{\alpha j m q} \int_Y \rho \gamma_{m\beta}~dy\Big) \nonumber \\
& -~ \frac{\partial^3 \widetilde{u}^\iss 1_{q}}{\partial x_\alpha \partial x_i \partial x_m}  \int_Y \big( C_{\alpha \beta k \ell}\frac{\partial \chi_{imq\ell}}{\partial y_k}  - C_{\alpha \beta i \ell} \chi_{ \ell m q} \big) ~dy. \nonumber \\
&+~ \omega^2 \frac{\partial \widetilde{u}^\iss 1_{q}}{ \partial x_m} \int_Y \big(\rho \chi_{\beta m q} -C_{m \beta k \ell} \frac{\partial \gamma_{q \ell} }{\partial y_k}  \big)~dy   \label{pde of average u_2}
\end{align}
Now let us recall that the solution $u^\epsilon$ to \eqref{System of differential equation} has the following anstaz
\begin{align*}
  u^{\epsilon}(x,y)=  u^\iss 0(x,y)+\epsilon u^\iss 1(x,y)+\epsilon^{2}u^\iss 2(x,y)+\cdots=\sum_{k=0}^{\infty}\epsilon^{k}u^\iss k(x,y).
 \end{align*}
Note that all the cell functions are $Y$-periodic and their averages over $Y$ are zero, then from equations \eqref{0- Homogenized Equation}, \eqref{homogenized tilde{u}1} and change of index, we can summarize the governing equations for $\overline{u}^\iss 0$ and $\overline{u}^\iss 1$ in $D$, where $\overline{u}^\iss 0$, $\overline{u}^\iss 1$ are the averages of $u^\iss 0$ and $u^\iss 1$  respectively in the unit cell $Y$ (recall that $u^\iss 0(x,y)=u^\iss 0(x)$ so that $\overline u^\iss 0=u^\iss 0$),

\begin{eqnarray}
\nabla\cdot(\overline{C}\nabla \overline{u}^\iss 0)+\omega^{2}\overline{\rho}   \overline{u}^{\iss 0} &=& 0,  \label{pde of average u_0} \\
  \overline{C}_{ijk\ell}  \frac{\partial^2 \overline{u}^\iss 1_{\ell}}{\partial x_i \partial x_k} + \omega^2 \overline{\rho} \overline{u}^\iss 1_j \label{pde of average u_1}
&=& \notag -\left( \frac{\partial^3  \overline{u}^\iss 0_{q}}{\partial x_i \partial x_m \partial x_n}\right) \int_Y \left( -C_{ijn\ell} \chi_{\ell m q} + C_{ijk\ell} \frac{\partial }{\partial y_k} \chi_{mnq \ell}  \right)dy  \\ 
&&   -\omega^2 \frac{\partial \overline{u}^\iss 0_{n}}{\partial x_m}  \int_Y\left( -\rho \chi_{jmn} + C_{mjk\ell} \frac{\partial }{\partial y_k} \gamma_{n \ell} \right)dy.
\end{eqnarray}
Let $U:= \overline{u}^\iss 0 + \epsilon \overline{u}^\iss 1 +  \epsilon^2 \overline{u}^\iss 2$. Now multiply equation \eqref{pde of average u_2} and equation \eqref{pde of average u_1}   by $\epsilon^2$ and $\epsilon$ respectively, and sum them with equation \eqref{pde of average u_0}, then it is seen that $U$ satisfies the following fourth-order partial differential equation

\begin{align*}
 \overline{C}_{\alpha \beta m q}  \frac{\partial^2 U^{q}}{\partial x_\alpha \partial x_m} + \omega^2 \overline{\rho} U ^\beta 
 =&  -\epsilon^2 \frac{\partial^4 u^\iss 0_{q}}{\partial x_\alpha \partial x_i \partial x_m \partial x_n}  \int_Y \big( C_{\alpha \beta k \ell}\frac{\partial \chi_{inmq\ell}}{\partial y_k}  + C_{\alpha \beta i \ell} \chi_{mnq\ell} \big) ~dy\\
& - ~ \epsilon^2 \omega^2 \frac{\partial^2 u^\iss 0_{q}}{\partial x_\alpha  \partial x_m} \Big(  \int_Y \big( C_{\alpha \beta k \ell}\frac{\partial \gamma_{\ell m q}}{\partial y_k}  + C_{\alpha \beta m \ell} \gamma_{q\ell} \big) ~dy \nonumber \\
& \qquad \qquad \qquad \quad + \int_Y \rho \chi_{m\alpha q \beta} ~dy - (\overline{\rho})^{-1}\overline{C}_{\alpha j m q} \int_Y \rho \gamma_{m\beta}~dy\Big) \nonumber \\
& -~ \epsilon \frac{\partial^3 (\epsilon \widetilde{u}^\iss 1_{q}+u^\iss 0_q)}{\partial x_\alpha \partial x_i \partial x_m}  \int_Y \big( C_{\alpha \beta k \ell}\frac{\partial \chi_{imq\ell}}{\partial y_k}  - C_{\alpha \beta i \ell} \chi_{ \ell m q} \big) ~dy. \nonumber \\
&+~ \epsilon \omega^2 \frac{\partial (\epsilon \widetilde{u}^\iss 1_{q}+ u^\iss 0_q)}{ \partial x_m} \int_Y\big(\rho \chi_{\beta m q} -C_{m \beta k \ell} \frac{\partial \gamma_{q \ell} }{\partial y_k}  \big) ~dy, \nonumber 
\end{align*}
and therefore the governing equation of $U$ up to order $\epsilon^3$ reads
\begin{align*}
 \overline{C}_{\alpha \beta m q}  \frac{\partial^2 U^{q}}{\partial x_\alpha \partial x_m} + \omega^2 \overline{\rho} U ^\beta  
 =&  -\epsilon^2 \frac{\partial^4 U^{q}}{\partial x_\alpha \partial x_i \partial x_m \partial x_n}  \int_Y \big( C_{\alpha \beta k \ell}\frac{\partial \chi_{inmq\ell}}{\partial y_k}  + C_{\alpha \beta i \ell} \chi_{mnq\ell} \big) ~dy\\
& - ~ \epsilon^2 \omega^2 \frac{\partial^2 U^{q}}{\partial x_\alpha  \partial x_m} \Big(  \int_Y \big( C_{\alpha \beta k \ell}\frac{\partial \gamma_{\ell m q}}{\partial y_k}  + C_{\alpha \beta m \ell} \gamma_{q\ell} \big) ~dy \nonumber \\
& \qquad \qquad \qquad \quad + \int_Y \rho \chi_{m\alpha q \beta} ~dy - (\overline{\rho})^{-1} \overline{C}_{\alpha j m q} \int_Y \rho \gamma_{m\beta}~dy\Big) \nonumber \\
& -~ \epsilon \frac{\partial^3 U^q}{\partial x_\alpha \partial x_i \partial x_m}  \int_Y \big( C_{\alpha \beta k \ell}\frac{\partial \chi_{imq\ell}}{\partial y_k}  - C_{\alpha \beta i \ell} \chi_{ \ell m q} \big) ~dy. \nonumber \\
&-~ \epsilon \omega^2 \frac{\partial U^q}{ \partial x_m} \int_Y\big(-\rho \chi_{\beta m q} +C_{m \beta k \ell} \frac{\partial \gamma_{q \ell} }{\partial y_k}  \big) ~dy + O(\epsilon^3). \nonumber 
\end{align*}
Hence the fourth-order equation can be conveniently casted as
\begin{align} \label{fourth-order PDE}
&\notag  \nabla \cdot \big( \overline{C} \nabla U \big) + \omega^2 \overline{\rho} U \\
=&-\epsilon^2 \big( D: \nabla^4 U +   \omega^2 E: \nabla^2 U \big)  - \epsilon \big( F : \nabla^3 U + \omega^2 G: \nabla U \big) + O(\epsilon^3),
\end{align}
where $D$ is a sixth-order tensor, $E$ is a fourth-order tensor, $F$ is a fifth-order tensor and $G$ is a third-order tensor respectively defined by 

\begin{eqnarray*}
D_{\beta \alpha i mn q} &=&  \int_Y \big( C_{\alpha \beta k \ell}\frac{\partial \chi_{inmq\ell}}{\partial y_k}  + C_{\alpha \beta i \ell} \chi_{mnq\ell} \big) ~dy, \\
E_{\beta \alpha m q} &=&  \int_Y \big( C_{\alpha \beta k \ell}\frac{\partial \gamma_{\ell m q}}{\partial y_k}  + C_{\alpha \beta m \ell} \gamma_{q\ell} \big) ~dy + \int_Y \rho \chi_{m\alpha q \beta} ~dy \\
&&- (\overline{\rho})^{-1} \overline{C}_{\alpha j m q} \int_Y \rho \gamma_{m\beta}~dy, \\
F_{\beta \alpha i m q} &=& \int_Y\big( C_{\alpha \beta k \ell}\frac{\partial \chi_{imq\ell}}{\partial y_k}  - C_{\alpha \beta i \ell} \chi_{ \ell m q} \big) ~dy, \\
G_{\beta m q} &=& \int_Y\big(-\rho \chi_{\beta m q} +C_{m \beta k \ell} \frac{\partial \gamma_{q \ell} }{\partial y_k}  \big) ~dy.
\end{eqnarray*}
The fourth-order partial differential equation \eqref{fourth-order PDE} in $D$ formally introduce the dispersion as is seen from the right hand side. In the low-frequency long-wavelength regime for wave propagation in periodic media, the dispersive wave equation has been demonstrated by a fourth-order partial differential equation for the acoustic case \cite{cakoni2016homogenization}. In the particular case that the periodic media occupies $\R^d$, \eqref{fourth-order PDE} models the wave propagation and transcends the quasi-static regime. To the authors' knowledge, our second-order homogenization for elastic wave is new in the literature.

\section{Appendix}\label{Section 5}

In the end of this paper, we offer basic materials in analysing the elastic scattering in periodic media.
\subsection{The Dirichlet to Neumann map}  \label{Appd_DtN}
Let $u$ satisfy the Navier's equation in the exterior domain 
\begin{equation*} \label{DtNNavier}
\Delta^{*}u+\omega^{2}u=0  \mbox{ in }\mathbb{R}^{d}\backslash\overline{D},
\end{equation*}
and $u$ has a decomposition that satisfies the \textit{Kupradze radiation condition}. Let $B_R$ be a sufficiently large ball such that $D \subset B_R$. In the case that $D \in \R^3$, we introduce the polar coordinates $r, \theta, \phi$ and the unit vectors $\hat{r}, \hat{\theta}, \hat{\phi}$. The $\theta$ coordinate corresponds to the angle from the $z$-axis, $\theta \in [0, \pi]$, and the $\phi$ coordinate corresponds to the angle in the $(x, y)$-plane, $\phi \in [0, 2\pi]$. Let $Y_{nm}$ be the spherical harmonic
\begin{equation*}
Y_{nm}(\theta, \phi) = \sqrt{ \frac{(2n+1)(n-|m|)!}{4\pi (n+|m|!)} } P_n^{|m|}(\cos \theta) e^{im\phi}, \quad n \ge 0, \quad |m|<n.
\end{equation*}
Now we let $U_{nm}$ and $V_{nm}$ be the vector spherical harmonics defined by
\begin{eqnarray*}
U_{nm} (\theta, \phi) &=& \frac{1}{\sqrt{\lambda_n}} \big(  \frac{\partial Y_{nm}}{\partial \theta} \hat{\theta} + \frac{1}{\sin \theta} \frac{\partial Y_{nm}}{\partial \phi} \hat{\phi} \big), \quad n\ge 1, \\
V_{nm} (\theta, \phi) &=& \hat{r} \times U_{nm} = \frac{1}{\sqrt{\lambda_n}} \big(  -\frac{1}{\sin \theta} \frac{\partial Y_{nm}}{\partial \phi} \hat{\theta} + \frac{\partial Y_{nm}}{\partial \theta} \hat{\phi} \big), \quad n\ge 1,
\end{eqnarray*}
where $\lambda_n = n(n+1)$.  The vectors $Y_{nm} \hat{r}, U_{nm}, V_{nm}$ form an orthonormal basis for $L^2(S)$ where $S$ denotes the unit sphere. Then $u$ on $\partial B_R$ has the following series expansion
\begin{eqnarray} \label{DtN_Dirichlet}
u = \sum_{n=0}^\infty \sum_{|m|<n} \Big(  (u|_{\partial B_R}, V_{nm}) V_{nm}  +  (u|_{\partial B_R}, U_{nm}) U_{nm}  + (u|_{\partial B_R}, Y_{nm} \hat{r}) Y_{nm} \hat{r}  \Big), 
\end{eqnarray}
where $(\cdot,\cdot)$ denotes the $L^2(S)$ inner product.
One can correspondingly express $T_\nu u$ on $\partial B_R$ as (see \cite{gachter2003dirichlet})
\begin{eqnarray}  \label{DtN_Neumann}
&T_\nu u = & \sum_{n=0}^\infty \sum_{|m|<n}   \Big(  a_n (u|_{\partial B_R}, V_{nm}) V_{nm}  + \big[ b_n  (u|_{\partial B_R}, U_{nm})  +  c_n (u|_{\partial B_R}, Y_{nm} \hat{r})  \big] U_{nm}  \\
&&\qquad \qquad + \big[ c_n  (u|_{\partial B_R}, u_{nm}) + d_n (u|_{\partial B_R}, Y_{nm} \hat{r}) \big] Y_{nm} \hat{r}  \Big) \quad \mbox{on} \quad B_R. \nonumber 
\end{eqnarray}
The coefficients $a_n, b_n, c_n, d_n$ are given by
\begin{eqnarray*}
a_n &=& \mu_0 (\gamma_s - \frac{1}{R}), \\
b_n &=&  \Big(  2 \mu_0 \sqrt{\lambda_n} (\gamma_p - \frac{1}{R}) \Big) \frac{B_n^{(1,1)}}{R} +  \mu_0 \Big( 2 \gamma_s + R \omega_s^2 + 2 (1-\lambda_n)  \frac{1}{R}   \Big) \frac{B_n^{(2,1)}}{R}, \\
c_n &=& \Big(  2 \mu_0 \sqrt{\lambda_n} (-\gamma_s + \frac{1}{R}) \Big) \frac{B_n^{(2,1)}}{R} + \Big(  2 \mu_0 (- 2\gamma_p + \frac{\lambda_n}{R}) \Big) \frac{B_n^{(1,1)}}{R-\mu_0 \omega_s^2}, \\
d_n &=& -\Big(  2 \mu_0 \sqrt{\lambda_n} (-\gamma_s + \frac{1}{R}) \Big) \frac{B_n^{(1,1)}}{R}  + \Big(  2 \mu_0 (- 2\gamma_p + \frac{\lambda_n}{R}) \Big) \frac{B_n^{(1,2)}}{R-\mu_0 \omega_s^2},
\end{eqnarray*}
where 
\begin{eqnarray*}
&& \gamma_s = \omega_s \frac{h_n^{'}(\omega_s R)}{h_n(\omega_s R)}, \quad \gamma_p = \omega_p \frac{h_n^{'}(\omega_p R)}{h_n(\omega_p R)}, \quad  B_n^{(1,1)} = -\frac{\sqrt{\lambda_n} R}{R \gamma_p (R\gamma_s + 1)-\lambda_n}, \\
&&B_n^{(1,2)} =  \frac{ R(1+ R \gamma_s) }{R \gamma_p (R\gamma_s + 1)-\lambda_n}, 
 \quad B_n^{(1,1)} = -\frac{R^2 \gamma_p}{R \gamma_p (R\gamma_s + 1)-\lambda_n}.
\end{eqnarray*}
Now for any functions $w$ and $u$ that satisfy the Kupradze radiation condition, one can directly obtain from \eqref{DtN_Dirichlet} and \eqref{DtN_Neumann} that
\begin{eqnarray*}
\int_{\partial B_R}T_\nu u \cdot w dS - \int_{\partial B_R}T_\nu w \cdot u dS =0.
\end{eqnarray*}
We remark that when $D \subset \R^2$, the above equality can be derived in a similar way \cite{bao2018direct}.

Let $B_R\subset \mathbb R^d$ be a ball of radius $R>0$, then the Dirichlet to Neumann (DN) map was given by \cite{bao2018direct}.
\begin{definition}
For any $g\in H^{1/2}(\partial B_R)$, the DN map
\begin{align}\label{DN map}
\Lambda:H^{1/2}(\partial B_R)\to H^{-1/2}(\partial B_R) \quad\text{ with }\quad \Lambda g|_{\partial B_R} = T_\nu u|_{\partial B_R}, 
\end{align}  
where $u\in H^1_{loc}(\mathbb R^d\setminus\overline{B_R})$ is a solution of the Navier's equation $\Delta ^*u+\omega ^2u=0$ in $\mathbb R^d \setminus \overline{B_R}$ and $u$ satisfies the Kupradze radiation condition \eqref{Kupradze radiation condition} at infinity. Here we have assumed that $\omega$ is non-resonant to the above Navier's equation.
\end{definition}
Notice that the DN map $\Lambda$ is a bounded operator, so that it helps to reduce the scattering problem in unbounded domain to a bounded domain, and we refer readers to \cite[Section 2]{bao2018direct} for detailed discussions.

\subsection{Derivation of the homogenized equation}
Consider the simplest linear elliptic system
of the homogenization theory. The periodic homogenization theory was studied by \cite{cioranescu2000introduction, geng2017boundary} and we refer readers to these references for the comprehensive study.
We are concerned with the divergence
form second order elliptic operators with rapidly oscillating periodic
coefficients, 
\[
\mathcal{L}_{\epsilon}:=-\nabla\cdot\left(A(\frac{x}{\epsilon})\nabla\right)=-\dfrac{\partial}{\partial x_{i}}\left(
	 a_{ijk\ell}(\dfrac{x}{\epsilon})\dfrac{\partial}{\partial x_{k}}\right),\quad\epsilon>0.
\]
We assume the coefficients $A(y)=\left(a_{ijk\ell}(y)\right)$
with $1\leq i,j,\alpha,\beta\leq d$ for the dimension $d\geq2$ is
real, bounded and measurable such that $A$ satisfies 
\begin{equation}\label{ellipticity}
\mbox{ellipticity: }\mu\sum _{i,j=1}^d|\varepsilon_{ij}|^{2}\leq a_{ijk\ell}(y)\varepsilon_{ij}\varepsilon_{k\ell}\leq\dfrac{1}{\mu}\sum _{i,j=1}^d|\varepsilon_{ij}|^{2},
\end{equation}
for all symmetric matrix $(\varepsilon_{ij})_{1\leq i,j\leq d}$, and 
\begin{equation*}\label{periodicity}
\mbox{Y-periodicity: }A(y+z)=A(y)\mbox{ for all }y\in\mathbb{R}^{d},\mbox{ }z \in Y:=[0,1]^d,
\end{equation*}
for some constant $\mu>0$.

Given $F\in H^{-1}(\Omega)$, let $u_{\epsilon}\in H_{0}^{1}(\Omega)$
be a solution of 
\begin{equation}
\mathcal{L}_{\epsilon}u^{\epsilon}=F\mbox{ in }\Omega,\label{Hom}
\end{equation}
where $\Omega$ is a bounded Lipschitz domain in $\mathbb{R}^{d}$.
By the Lax-Milgram theorem, we have 
\begin{equation*}
\|u^{\epsilon}\|_{H_{0}^{1}(\Omega)}\leq C\|F\|_{H^{-1}(\Omega)},\label{eq:Lax_MilgramD}
\end{equation*}
where the constant $C$ independent of $\epsilon$. Note that $u_{\epsilon}\in H_{0}^{1}(\Omega)$
is a weak solution of (\ref{Hom}) if for all $\varphi\in H_{0}^{1}(\Omega)$,
we have 
\[
\int_{\Omega}(A(\dfrac{x}{\epsilon})\nabla u^{\epsilon}):\nabla\varphi dx=\left\langle F,\varphi\right\rangle _{H^{-1}(\Omega)\times H_{0}^{1}(\Omega)}.
\]

Next, we want to derive the homogenized equation by using the following
asymptotic analysis. We consider $u_{\epsilon}$ to be the perturbation
of $u_{0}$ with respect to $\epsilon$-parameter. Moreover, by observing
the elliptic operator $\mathcal{L}_{\epsilon}$, we introduce the
famous \textit{two-scale homogenization method} in the homogenization theory: Let
us regard $x=x$, and $y=\dfrac{x}{\epsilon}$ as two independent
parameters. Let 
\[
u^{\epsilon}:=u^\iss 0+\epsilon u^{(1)}+\epsilon^{2}u^\iss 2+\cdots
\]
be the asymptotic expansion of $u_{\epsilon}$, where 
\[
u^\iss j:=u^\iss j(x,y)=u^\iss j(x,\dfrac{x}{\epsilon}).
\]
In addition, 
\[
\nabla u^\iss j=\nabla_{x}u^\iss j(x,y)+\cfrac{1}{\epsilon}\nabla_{y}u^\iss j(x,y),\mbox{ as }y=\dfrac{x}{\epsilon},
\]
which means under our two-scaled method, the operator $\nabla=\nabla_{x}+\cfrac{1}{\epsilon}\nabla_{y}$.
Therefore, (\ref{Hom}) will become
\begin{equation}
-\left(\nabla_{x}+\dfrac{1}{\epsilon}\nabla_{y}\right)\cdot\left\{A(y)\left[\left(\nabla_{x}+\cfrac{1}{\epsilon}\nabla_{y}\right)\left(u^\iss 0+\epsilon u^\iss 1+\epsilon^{2}u^\iss 2+\cdots\right)\right]\right\}=F(x)\mbox{ in }\Omega.\label{eq:2SCALE}
\end{equation}
We point out that the derivation of the homogenized equation did not need to take care of the boundary condition of certain equations.
Expand (\ref{eq:2SCALE}) and compare it with the same $\epsilon^N$-orders (for $N=0,-1,-2$), so
we get
\begin{alignat}{1}\label{Eq of Asym}
\notag O(\dfrac{1}{\epsilon^{2}}):\mbox{ }-\nabla_{y}\cdot(A(y)\nabla_{y}u^\iss 0(x,y)) & =0,\\
O(\dfrac{1}{\epsilon}):\mbox{ }-\nabla_{y}\cdot(A(y)\nabla_{y}u^\iss 1(x,y)) & =\nabla_{y}\cdot(A(y)\nabla_{x}u^\iss 1)+\nabla_{x}\cdot(A(y)\nabla_{y}u^\iss 0),\\
\notag 
O(1):\mbox{ }-\nabla_{y}\cdot(A(y)\nabla_{y}u^\iss 2(x,y)) & =\nabla_{y}\cdot(A(y)\nabla_{x}u^\iss 1)+\nabla_{x}\cdot(A(y)\nabla_{y}u^\iss 1)\\
&\notag +\nabla_{x}\cdot(A(y)\nabla_{x}u^\iss 0)+F(x). 
\end{alignat}
Recall that for the periodic elliptic equation 
\[
-\nabla\cdot(A(y)\nabla v(y))=h(y),\mbox{ whenever }A(y)\mbox{ is Y-periodic,}
\]
then we have 
\[
\int_{Y}h(y)dy=0,
\]
by using the Stokes formula. For $O(\dfrac{1}{\epsilon^{2}})$
term, this equation is solvable because the right hand side is zero.
In further, we multiply $u^\iss 0(x,y)$ on both sides and integrate
by parts, which will imply 
\[
0=\int_{Y} \big( A(y)\nabla_{y}u^\iss 0 \big):\nabla_{y}u^\iss 0\geq\mu\int_{Y}|\nabla_{y}u^\iss 0(x,y)|^{2}dy\geq0,
\]
which gives us the information that 
\[
u^\iss 0(x,y)\equiv u^\iss 0(x)
\]
and $u_{0}$ is independent of $y$. 

Now, for the second term $O(\dfrac{1}{\epsilon})$, the second term
on the right hand side should be zero since $\nabla_{y}u^\iss 0(x)=0$.
Solve the equation 
\begin{align*}
-\nabla_{y}\cdot(A(y)\nabla_{y}u^\iss 1(x,y)) =\nabla_{y}\cdot(A(y)\nabla_{x}u^\iss 0)=(\nabla_{y}\cdot A(y))(\nabla_{x}u^\iss 0)
\end{align*}
formally. Note that since $A(y)$ is $Y$-periodic, then the equation
is solvable for $u^\iss 1$ if 
\[
\int_{Y}(\nabla_{y}\cdot A(y))\cdot(\nabla_{x}u^\iss 0)dy=\int_{\partial Y}(A(y)\nabla_{x}u^\iss 0)\cdot\nu(y)dS(y)=0.
\]
By using the separation of variables, we put the \textit{ansatz} 
\[
u^\iss 1(x,y)=\chi(y)\cdot (\nabla_{x}u^\iss 0(x))
\]
with $u^\iss 1=(u^\iss 1_{\alpha})_{1\leq\alpha\leq d}$ such that 
\[
u^\iss 1_{\alpha}(x,y)=\chi_{\alpha j \beta }(y)\dfrac{\partial u^\iss 0_{\beta}}{\partial x_{j}}(x).
\]
Moreover, the corrector $\chi_{\alpha j \beta}$
is $Y$-periodic and solves the cell problem 
\[
\begin{cases}
\dfrac{\partial}{\partial y_{i}}\left(a_{ijmn}-a_{ijk\ell}\dfrac{\partial}{\partial y_{k}} \chi_{\ell m n}\right)=0\mbox{ in } Y,\\
\int_{Y}\chi_{\ell mn}(y)dy=0,
\end{cases}
\]
and plug $u^\iss 1$ to the $O(\dfrac{1}{\epsilon})$ equation \eqref{Eq of Asym} to obtain 
\[
-\nabla\cdot(A(y)\nabla_{y}\chi(y))(\nabla_{x}u^\iss 0)=(\nabla_{y}\cdot A(y))(\nabla_{x}u^\iss 0).
\]

Finally plug $u^\iss 1(x,y)=\chi(y)\nabla_{x}u^\iss 0$
into the $O(1)$ equation and examine the solvability condition for
$u^\iss 2(x,y)$, we have 
\begin{align*}
0 & =\int_{Y}\left[\nabla_{y}\cdot(A(y)\nabla_{x}u^\iss 1)+\nabla_{x}\cdot(A(y)\nabla_{y}u^\iss 1)+\nabla_{x}\cdot(A(y)\nabla_{x}u^\iss 0)+F(x)\right]dy\\
& =\nabla_{x}\cdot\left\{\left[\int_{Y}A(y)(\nabla_{y}\chi(y))dy\right]\nabla_{x}u^\iss 0\right\}+\nabla_{x}\cdot\left\{\left[\int_{Y}A(y)dy\right]\nabla_{x}u^\iss 0\right\}+F(x),
\end{align*}
where the first term vanishes by the periodicity of $A$ and $\chi$.
Thus, we can obtain that $u^\iss 0\in H_{0}^{1}(\Omega)$
is a solution of 
\begin{equation}
\overline{\mathcal{L}}u^\iss 0:=-\nabla\cdot(\overline{A}\nabla u^\iss 0)=F(x)\mbox{ in }\Omega,\label{eq:HOMOGE}
\end{equation}
where 
\[
\overline{A}=\int_{Y}\left\{A(y)+A(y)(\nabla_{y}\chi(y))\right\}dy,
\]
where $\overline{A}$ is the (constant) homogenized operator and we
call (\ref{eq:HOMOGE}) to be the \textit{homogenized equation}. In
addition, $\overline{A}=(\overline{a}_{ijk\ell})_{1\leq i,j,k,\ell \leq d}$ and 
\begin{equation}
\overline{a}_{ijk\ell}=\int_{Y}\left(a_{ijk\ell}-a_{ijmn}\dfrac{\partial}{\partial y_{m}} \chi_{nk\ell}\right)dy.\label{Homogenized leading order}
\end{equation}
For the rigorous derivation of the homogenized equation, we need to
use a famous result, which is called the Div-Curl lemma. We skip the
rigorous analysis here and refer readers to the lecture note  \cite{shen2017lectures} for more
details.

Note that $\overline{\mathcal{L}}:=-\nabla\cdot(\overline{A}\nabla)$
is the homogenized second order elliptic operator with respect to
$A$ and we want to prove $\overline{\mathcal{L}}$ is an elliptic
operator with constant coefficients.
\begin{theorem}
	The homogenized operator $\overline{\mathcal L}$ satisfies that 
	
	1. $\overline{\mathcal{L}}$ is an elliptic operator, which means 
	\begin{equation}
	\mu_1\sum _{i,j=1}^d|\varepsilon_{ij}|^{2}\leq \overline{a}_{ijk\ell}(y)\varepsilon_{ij}\varepsilon_{k\ell}\leq\dfrac{1}{\mu_1}\sum _{i,j=1}^d|\varepsilon_{ij}|^{2}, \label{ellipticity}
	\end{equation}
	for some constant $\mu_{1}>0$.\end{theorem}

    2. The effective coefficient $\overline{a}_{ijk\ell}$ is major and minor symmetric provided $a_{\alpha \beta \gamma \delta }$ is major and minor symmetric.
\begin{proof}
	It is easy to see that $|\overline{a}_{ijk\ell}|\leq C$
	by using \eqref{Homogenized leading order} and the ellipticity of
	$A(y)$, for some constant $C>0$. It remains to show $\overline{a}_{ijk\ell}\varepsilon_{ij}\varepsilon_{k\ell}\geq\mu_1\sum _{i,j=1}^d|\varepsilon_{ij}|^{2}$ for some constant $\mu _1>0$.
	We can rewrite \eqref{Homogenized leading order} as 
	\[
	\overline{a}_{ijk\ell}=\int_{Y}\dfrac{\partial}{\partial y_{\alpha}}\left\{ \delta_{\beta j}y_{i}+\chi_{\beta ij}\right\} \cdot a_{\alpha \beta\gamma \delta}\cdot\dfrac{\partial}{\partial y_{\gamma}}\left\{ \delta_{\delta \ell}y_{k}+\chi_{\delta k \ell}\right\} dy,
	\]
	where $\delta_{s\alpha}$ is the standard Kronecker delta (i.e., $\delta_{s\alpha}=1$
	if $s=\alpha$, and $\delta_{s\alpha}=0$ otherwise). Hence, for $\varepsilon=(\varepsilon_{ij})\in\mathbb{R}^{d\times d}$,
	we have 
	\begin{align*}
	\overline{a}_{ijk\ell}\varepsilon_{ij}\varepsilon_{k\ell}&=\int_{Y}\dfrac{\partial}{\partial y_{\alpha}}\left\{ \delta_{\beta j}y_{i}\varepsilon_{ij}+\chi_{\beta ij}\varepsilon_{ij}\right\} \cdot a_{\alpha \beta\gamma \delta}\cdot\dfrac{\partial}{\partial y_{\gamma}}\left\{ \delta_{\delta \ell}y_{k}\varepsilon_{k\ell}+\chi_{\delta k \ell}\varepsilon_{k\ell}\right\} dy\\
	& \geq\mu\sum_{\beta=1}^{d}\int_{Y}|\nabla(y_{i}\varepsilon_{i\beta}+\chi_{\beta ij}\varepsilon_{ij})|^{2}dy\geq0.
	\end{align*}
	If $\overline{a}_{ijk\ell}\varepsilon_{ij}\varepsilon_{k\ell}=0$
	for some  $\varepsilon=(\varepsilon_{ij})\in\mathbb{R}^{d\times d}$, then
	$y_{i}\varepsilon_{i\beta}+\chi_{\beta ij}$ must be a constant.
	Recall that $\chi_{\beta ij}(y)$ is $Y$-periodic, so this implies
	that $\varepsilon=0$. This means that there exists $\mu_{1}>0$ such that
	\eqref{ellipticity} holds.\end{proof}

\subsection{Tools and estimates}
In the last part, for the completeness of this paper, we provide some elliptic estimate where we have utilized in previous sections. The following theorem was proved in \cite[Theorem 5.7]{cakoni2014qualitative} for the scalar case. It will hold for the vector case. For completeness, we provide the theorem and its proof as follows.

\begin{theorem}\label{trace theorem}(Trace Theorem)
Let  $A=(a_{ijk\ell})_{1\leq i,j,k,\ell\leq d}$ be a four tensor satisfying the ellipticity condition \eqref{ellipticity} and $D\subset \mathbb R^d$ be a bounded domain with $C^\infty$-smooth boundary, for $d\geq 2$. The (cornormal) mapping $Tr:u\to \dfrac{\partial u}{\partial \nu _A}:=(A\nabla u)\cdot\nu$ defined in $C^\infty(\overline{D})$ can be continuously extended to a linearly continuous mapping (still denote by $Tr$) from $H^1(D,A)$ to $H^{-1/2}(\partial D)$, where $H^1(D,A)$ is the space equipped with the graph norm
\begin{equation*}
\|u \|_{H^1(D,A)}^2:=\|u\|_{H^1(D)}^2 +\|\nabla \cdot (A\nabla u)\|_{L^2(D)}^2.
\end{equation*}
\end{theorem}
\begin{proof}
Let $\varphi\in C^\infty (\overline{D})$ be a test function and $u\in C^\infty (\overline{D})$. The integration by parts formula gives 
\begin{align*}
\int _{\partial D}(A\nabla u\cdot \nu) \cdot\varphi dS=\int _D(A\nabla u):\nabla \varphi dx+\int _D\nabla \cdot (A\nabla u)\cdot \varphi dx.
\end{align*}
By the standard density arguments, the above equation holds for $\varphi \in H^1(D)$ so that 
\begin{align}\label{some inequality}
\left|\int _{\partial D}(A\nabla u\cdot \nu) \cdot\varphi dS\right|\leq C \|u\|_{H^1(D,A)}\|\varphi \|_{H^1(D)}, \text{ for any }\varphi \in H^1(D),\ u\in C^\infty (\overline{D}),
\end{align}
for some constant $C>0$ independent of $\varphi$ and $u$. Let $g\in H^{1/2}(\partial D)$, by using the trace theorem, then there exists a function $\varphi \in H^1(D)$ such that $\gamma _{\partial D}\varphi =f$, where $\gamma _{\partial D}$ stands for the trace operator. Continuing the inequality \eqref{some inequality} and the trace theorem, 
\begin{align*}
\left|\int _{\partial D}(A\nabla u\cdot \nu) \cdot f dS\right|\leq C \|u\|_{H^1(D,A)}\|f \|_{H^{1/2}(\partial D)}, \text{ for any }f \in H^{1/2}(\partial D),\ u\in C^\infty (\overline{D}).
\end{align*}
Hence, the mapping 
\begin{equation}\notag
f \to \int_{\partial D}(A\nabla u\cdot \nu)\cdot f dS,\text{ for any }f\in H^{1/2}(\partial D)
\end{equation}
defines a continuous linear operator and from the duality argument,
\begin{align}\notag
\|(A\nabla)\cdot \nu\|_{H^{-1/2}(\partial D)}\leq C\|u \|_{H^1(D,A)}.
\end{align}
Therefore, the linear mapping $\gamma: u \to (A\nabla u)\cdot \nu $ is defined on $C^\infty(\overline{D})$ is continuous under the norm $H^1(D,A)$. Thus, the assertion follows from the density arguments.
\end{proof}
Let $C=(C_{ijk\ell})$ be an anisotropic elastic four tensor and $C_0$ be a constant isotropic elastic tensor defined by \eqref{constant elastic tensor}, which satisfy all the conditions given in Section \ref{Section 1}. Next, we provide the stability estimate for the following transmission problem. The scalar case was demonstrated in \cite[Section 5]{cakoni2014qualitative} and here we generalize the result to a system version.

\begin{theorem}\label{estimate on transmission problem}
Let $D\subset \mathbb R^d$ be a bounded $C^\infty $-smooth domain. Given $f\in H^{1/2}(\partial D)$ and $g\in H^{-1/2}(\partial D)$. Let $u\in H^1(D)$ and $v\in H^1_{loc}(\mathbb R^d\setminus \overline{D})$ be the solutions of the following transmission problem 
\begin{align}
\begin{cases}\label{H^1 estimates for transmission problem}
\nabla \cdot \left(C\nabla u\right)+\omega ^2 \rho u=0&\text{ in }D,\\
\Delta ^*v+\omega ^2 v=0&\text{ in } \mathbb R^d \setminus \overline{D},\\
u-v=f&\text{ on }\partial D,\\
T_\nu u-(C\nabla u)\cdot \nu=g&\text{ on }\partial D,
\end{cases}
\end{align}
where $T_\nu $ is the boundary traction operator given by  \eqref{boundary traction operator}, 
$\omega \in \mathbb R$ is not an eigenvalue of the transmission problem \eqref{H^1 estimates for transmission problem} and $v$ satisfies the Kupradze radiation condition \eqref{Kupradze radiation condition}. Then for any ball $B_R$ with $D\subset B_R$, there exists a constant $C_R>0$ such that 
\begin{align}\label{stability estimate}
\|u\|_{H^1(D)}+\|v\|_{H^1(B_R \setminus \overline{D})}\leq C_R\left\{\|f\|_{H^{1/2}(\partial D)}+\|g\|_{H^{-1/2}(\partial D)}\right\}.
\end{align}
\end{theorem}
\begin{proof}
Firstly, by using similar arguments in \cite[Section 2]{bao2018direct} and \cite[Section 5]{cakoni2014qualitative}, the elastic scattering problem \eqref{H^1 estimates for transmission problem} is equivalent to the following transmission problem: Let $u\in H^1(D)$ and $v\in H^1(B_R\setminus\overline{D})$ be the solutions of 
\begin{align}
\begin{cases}\label{bounded transmission problem}
\nabla \cdot \left(C\nabla u\right)+\omega ^2 \rho u=0 &\text{ in }D,\\
\Delta ^*v+\omega ^2 v=0 &\text{ in } B_R \setminus \overline{D},\\
u-v=f &\text{ on }\partial D,\\
T_\nu u-(C\nabla u)\cdot \nu=g &\text{ on }\partial D,\\
T_\nu v = \Lambda v &\text{ on }\partial B_R,
\end{cases}
\end{align}
where $\Lambda$ is the DN map defined by \eqref{DN map} on $\partial B_R$. Furthermore, by using \cite[Lemma 2.8]{bao2018direct}, the DN map $\Lambda$ is a bounded operator and $\Lambda$ can decomposed into  $\Lambda=\Lambda_1+\Lambda_2$, where $-\Lambda_1$ is a positive operator and $\Lambda_2$ is a compact operator from $H^{1/2}(\partial B_R)$ to $H^{-1/2}(\partial B_R)$. 

Next, let $v_f\in H^1(B_R\setminus\overline{D})$ be the unique solution of the Navier's equation in the exterior domain 
\begin{equation}
\begin{cases}\label{v_f}
\Delta^*v_f+\omega^2v_f=0& \text{ in }B_R \setminus \overline{D},\\
v_f=f &\text{ on }\partial D,\\
v_f =0 &\text{ on }\partial B_R.
\end{cases}
\end{equation}
By straight forward calculation, it is not hard to see that the variational formula of \eqref{bounded transmission problem} can be written as follows: Find a function $w\in H^1(B_R)$ such that  
\begin{align}\label{estimate 7}
&\notag\int_D \left((C\nabla w):\nabla \phi -\omega ^2\rho w\cdot \phi\right)\,dx+\int _{B_R\setminus\overline{D}}\left((C_0\nabla w):\nabla \phi-\omega^2 w\cdot \phi \right)\,dx\\
&\notag-\int _{\partial B_R}\phi \cdot T_\nu w\, dS+\int _{\partial B_R} \phi \cdot T_\nu v_f\,dS\\
=&\int_{\partial D}g\cdot \phi\, dS+\int _{B_R\setminus \overline{D}}\left((C_0\nabla v_f):\nabla \phi-\omega ^2v_f\cdot \phi \right)\,dx,
\end{align}
for any test function $\phi \in H^1(B_R)$, where $C_0$ is a constant elastic tensor defined by \eqref{constant elastic tensor}. By using the integration by parts, one can easily see that $u=w|_D$ and $v=w|_{B_R\setminus\overline{D}}-v_f$ satisfy \eqref{bounded transmission problem}. 

Now, let us consider two bilinear forms 
\begin{align*}
b_1(\psi,\phi):=&\notag\int_D \left((C\nabla \psi):\nabla \phi + \psi\cdot \phi\right)\,dx+\int _{B_R\setminus\overline{D}}\left((C_0\nabla w):\nabla \phi+w\cdot \phi \right)\,dx\\
&-\int_{\partial B_R}\phi \cdot (\Lambda_1\psi)\,dS,\\
b_2(\psi, \phi ):=&-\int_D\left(\omega^2\rho +1\right)\phi\cdot\psi dx-\int_{B_R\setminus\overline{D}}(\omega^2+1)\phi \cdot \psi dx\\
&-\int_{\partial B_R}\phi \cdot (\Lambda_2\psi)\, dS, \quad\text{ for all }\phi,\psi \in H^1(B_R),
\end{align*}
and 
\begin{align*}
F(\phi):=\int_{\partial D}g\cdot \phi\, dS-\int_{\partial B_R}\phi\cdot T_\nu v_f\,dS+\int _{B_R\setminus \overline{D}}\left((C_0\nabla v_f):\nabla \phi-\omega ^2v_f\cdot \phi \right)\,dx.
\end{align*}
Then we can rewrite the problem \eqref{estimate 7} as finding a function $w\in H^1(B_R)$ such that 
\begin{align*}
b_1(w,\phi)+b_2(w,\phi)=F(\phi),\quad\text{ for any }\phi \in H^1(B_R).
\end{align*}
Since $-\Lambda_1$ is a positive operator, one can conclude that $b_1(\cdot,\cdot)$ is strictly coercive. Therefore, from the Lax-Milgram theorem, one can see that the operator $A:H^1(B_R)\to H^1(B_R)$ defined by $b_1(w,\phi)=(Aw,\phi)_{H^1(B_R)}$ is invertible and has a bounded inverse.  On the other hand, since $\Lambda_2$ is a compact operator from $H^{1/2}(\partial B_R)\to H^{-1/2}(\partial B_R)$ and $H^{1}(B_R)\to L^2(B_R)$ is a compact embedding, then it is not hard to see that the operator $B:H^1(B_R)\to H^1(B_R)$ defined by $b_2(w,\phi)=(Bw,\phi)_{H^1(B_R)}$ is compact. Hence, by using \cite[Theorem 5.16]{cakoni2014qualitative}, one can derive that the existence of the transmission problem \eqref{bounded transmission problem} from the uniqueness of \eqref{bounded transmission problem}  and the stability estimate \eqref{stability estimate} holds automatically.
\end{proof}

\bibliographystyle{plain}
\bibliography{ref}

\end{document}